\documentclass[a4paper,12pt]{article}
\usepackage{authblk}
\usepackage{amsmath}
\usepackage{amsfonts}
\usepackage{amssymb}
\usepackage{amsthm}
\usepackage{graphicx}
\usepackage{appendix}
\usepackage{tabularx}
\usepackage{diagbox}
\usepackage{lscape}

\newcolumntype{L}[1]{>{\raggedright\arraybackslash}p{#1}} 	
\newcolumntype{C}[1]{>{\centering\arraybackslash}p{#1}} 	
\newcolumntype{R}[1]{>{\raggedleft\arraybackslash}p{#1}} 	

\DeclareMathOperator{\tr}{tr}

\newcommand{\R}{\mathbb{R}}

\newcommand{\id}{1\!\!1}

\newtheorem{theorem}{Theorem}

\newtheorem{remark}[theorem]{Remark}

\theoremstyle{remark}

\begin{document}
\title{Existence theorem for geometrically nonlinear Cosserat micropolar model under uniform convexity requirements}
\author{Patrizio Neff\,%
\thanks{Patrizio Neff, Head of Lehrstuhl f\"ur Nichtlineare Analysis und Modellierung, Fakult\"at f\"ur  Mathematik, Universit\"at Duisburg-Essen, Campus Essen, Thea-Leymann Str. 9, 45127 Essen, Germany, email: patrizio.neff@uni-due.de, Tel.: +49-201-183-4243}\,\,,\,
 Mircea B\^{\i}rsan%
\thanks{\,Mircea B\^{\i}rsan, Lehrstuhl f\"ur Nichtlineare Analysis und Modellierung, Fakult\"at f\"ur  Mathematik, Universit\"at Duisburg-Essen, Campus Essen, Thea-Leymann Str. 9, 45127 Essen, Germany, email: mircea.birsan@uni-due.de ; and
Department of Mathematics, University ``A.I. Cuza'' of Ia\c{s}i, 700506 Ia\c{s}i,  Romania}\,\,\,
and\,\,\,Frank Osterbrink\,%
\thanks{\,Frank Osterbrink, Lehrstuhl f\"ur Nichtlineare Analysis und Modellierung, Fakult\"at f\"ur  Mathematik, Universit\"at Duisburg-Essen, Campus Essen, Thea-Leymann Str. 9, 45127 Essen, Germany, email: frank.osterbrink@uni-due.de}
}


\maketitle

\begin{abstract}
We reconsider the geometrically nonlinear Cosserat model for a uniformly convex elastic energy and write the equilibrium system as a minimization problem. Applying the direct methods of the calculus of variations we show the existence of minimizers. We present a clear proof based on the coercivity of the elastically stored energy density and on the weak lower semi-continuity  of the total energy functional. Use is made of the dislocation density tensor $\overline{\boldsymbol{K}}=\overline{\boldsymbol{R}}^T\,\mathrm{Curl}\,\overline{\boldsymbol{R}}$ as a suitable Cosserat curvature measure.
\end{abstract}

\textbf{Keywords:} Cosserat continuum, geometricallynonlinear micropolar elasticity, existence theorem, \newline minimizers.

\section{Introduction}\label{sec1}

The Cosserat model for elastic continua is a generalized model of elastic bodies which is appropriate for deformable solids with a certain microstructure, such as defective elastic crystals, cellular materials, foam-like structures (e.g., bones), granular solids, periodic lattices and others. In the Cosserat continuum (also called \emph{micropolar} continuum) any material point is allowed to rotate without stretch, like an infinitesimal rigid body. Thus, in addition to the deformation field $\,\boldsymbol{\varphi}\,$, an independent rotation field $ \,\overline{\boldsymbol{R}}\,$ (also called \emph{microrotation}) is needed to fully describe the generalized continuum. The Cosserat model has therefore 6 degrees of freedomm (3 for the displacements and 3 for the rotations of points), in contrast to the classical elastic continuum, which has only 3 degrees of freedom.

From a historical perspective, the kinematical model of such generalized continua was introduced by the Cosserat brothers in \cite{Cosserat09}, but they did not provide any constitutive relations. Fifty years later, Ericksen and Truesdell \cite{Ericksen-Truesdell-58} have developed this idea and have drawn anew the attention to the theory of generalized continua. In the 1960s, several variants of the theory of media with microstructure were proposed by Toupin \cite{Toupin64}, Mindlin \cite{Mindlin64}, Eringen \cite{Eringen67} and others. Thus, the micropolar (or Cosserat) continua can be viewed as a special case of the microstretch continua and of the micromorphic continua (see Eringen \cite{Eringen98} for the foundation of this theory). In the last decades, the generalized theories of continua have proved useful for the treatment of complex mechanical problems, for which the classical theory of elasticity is  not satisfactory. In this respect, we mention for instance the works of Lakes \cite{Lakes93,Lakes93b,Lakes98}, Neff \cite{Neff_Cosserat_plasticity05}, Neff and Forest \cite{Neff_Forest07}, among others.

In parallel to the three-dimensional theory, the Cosserat model has contributed substantially to the development of dimensionally-reduced theories for shells, plates and rods. Ever since the beginning, the Cosserat brothers have indicated the advantages of their approach for the modeling of shells and rods \cite{Cosserat09}. This idea was used later to develop the so-called Cosserat theories for shells  and rods, see e.g., the works of  Ericksen and Truesdell \cite{Ericksen-Truesdell-58}, Eringen \cite{Eringen67}, Green and Naghdi \cite{Naghdi72}, Rubin \cite{Rubin00}, Zhilin \cite{Zhilin76}, Neff \cite{Neff_plate04_cmt,Neff_plate07_m3as}. We mention here that the general nonlinear theory of shells (derived independently from three-dimensional classical elasticity) is a 6-parameter shell theory, which kinematical model is equivalent to the Cosserat surface model. In this respect, see the books \cite{Libai98,Pietraszkiewicz-book04,Zhilin06} and the papers of Eremeyev and Pietraszkiewicz \cite{Pietraszkiewicz04,Eremeyev06} and B\^ irsan and Neff \cite{Birsan-Neff-JElast-2013,Birsan-Neff-MMS-2014}.

The mathematical problem related to the deformation of Cosserat elastic bodies has been investigated in many works. In linearized micropolar elasticity the existence and properties of solutions have been studied by Ie\c san \cite{Iesan70,Iesan71,Iesan82} and by Neff \cite{Neff_zamm06,Neff_JeongMMS08} under some weaker assumptions, among others. For the geometrically nonlinear micromorphic model, the first existence theorem based on convexity arguments was presented by Neff \cite{Neff_Edinb06,Neff-PAMM-2004,Neff_Habil04}. Another existence result for generalized continua with microstructure was proved in \cite{Mariano08}, under certain convexity assumptions. Adapting the methods from \cite{Ball77,Dacorogna89}, Tamba\v ca and Vel\v ci\'c \cite{Tambaca10} have derived an existence theorem for nonlinear micropolar elasticity, for a general constitutive behavior.

In our paper, we investigate the equilibrium problem for nonlinear Cosserat continua, in its minimization formulation. We prove the existence of global minimizers by applying the direct methods of the calculus of variations. Since the Cosserat model is a special case of the general micromorphic model, the theorem presented herein can be derived from the results in \cite{Neff_Edinb06}. Moreover, the first existence theorem for the geometrically nonlinear Cosserat model has been obtained in the first authors habilitation thesis already in 2003 \cite{Neff_Habil04,Neff-PAMM-2004}. However, these results seem to have gone mostly unnoticed. Therefore we find it expedient to discuss here the Cosserat case individually. In the micropolar case, we prove the existence of solutions for all exponents $\,p\geq 2$ in the curvature energy. Thus, we do not exclude here the important case of quadratic stored energy functions. On the other hand, the technique used in \cite{Tambaca10} has the drawback that the condition $\,p>3\,$ on the space $\,W^{1,p}(\Omega)$ where we look for solutions, is imposed. Thus, our theorem (which holds for $\,p\geq 2$) cannot be obtained as a consequence of the results in \cite{Tambaca10}.

Moreover, in the present paper we consider a different form of the curvature tensor, namely the dislocation density tensor $\overline{\boldsymbol{K}}= \,\overline{\boldsymbol{R}}^T\mathrm{Curl}\,\overline{\boldsymbol{R}}\,$, which seems to be more appropriate for Cosserat continua for three reasons: firstly, we obtain immediately a second order curvature tensor without artificial construction, secondly we have the natural connection to dislocation theory and thirdly the operator $\overline{\boldsymbol{K}}$ is the correct object for the group of rotations $\mathrm{SO}(3)$ (see \cite{Neff_curl08})

In Section \ref{sec2} we introduce the strain measures and curvature strain measures associated to the nonlinear Cosserat model. The next section presents the expression of the elastically stored energy density and the main assumptions on the constitutive coefficients.

The main part of the paper (Section \ref{sec4}) is devoted to the existence result. We present a clear and detailed proof of the theorem, based on convexity arguments. The coercivity of the stored energy density and the sequential weak lower semi-continuity of the energy functional play also important roles in the proof. We extend our investigation to chiral materials (Section \ref{sec5}). In the appendix we provide additional relations between the different curvature tensors and give a table summarizing our findings.

\section{Strain and curvature measures for the nonlinear Cosserat model}\label{sec2}

We consider an elastic body $\,\mathcal{B}\,$ which occupies in the reference configuration a bounded domain $\,\Omega\in\mathbb{R}^3$ with Lipschitz boundary $\partial\Omega$. Let $ Ox_1x_2x_3$ be a Cartesian coordinate frame in $\mathbb{R}^3$ and $\boldsymbol{e}_i$ the unit vectors along the coordinate axes $Ox_i\,$. We employ the usual conventions: boldface letters denote vectors and tensors; The Latin indices $i,j,k,...$ range over the set $\{1,2,3\}$; the comma preceding an index $i$ denotes partial derivatives with respect to $x_i$ (e.g., $f,_i=\frac{\partial f}{\partial x_i}\,$); the Einstein summation convention over repeated indices is also used.

We employ a \emph{Cosserat model} for the elastic body $\,\mathcal{B}\,$ under consideration. Such an elastic continuum is also called \emph{micropolar}. This model takes into consideration not only the displacements of material points, but also the rigid rotations of material points, which are often called \emph{microrotations}. Thus, every material point has 6 degrees of freedom (3 for translations and 3 for rotations).

Consider that the elastic body deforms under the action of some external loads and it reaches an equilibrium state. We denote the (macroscopic) \emph{deformation vector} with
$$\boldsymbol{\varphi}=\boldsymbol{\varphi}(x_1,x_2,x_3),\qquad \boldsymbol{\varphi}:\Omega\rightarrow\mathbb{R}^3$$
and the \emph{microrotation tensor} (describing the rotation of each material point) with
$$\overline{\boldsymbol{R}}=\overline{\boldsymbol{R}}(x_1,x_2,x_3),\qquad \overline{\boldsymbol{R}}:\Omega\rightarrow\mathrm{SO(3)}.$$
By abuse of notation, we identify the second order tensor $\,\overline{\boldsymbol{R}}\,$ with the $3\times 3$ matrix of its components $\,\overline{\boldsymbol{R}}\in \mathbb{R}^{3\times 3}$ and the vector $\,\boldsymbol{\varphi}\,$ with the column-vector of its components $\,\boldsymbol{\varphi}\in\mathbb{R}^3$. Notice that $\,\overline{\boldsymbol{R}}\,$ is a proper orthogonal tensor, i.e. $\,\overline{\boldsymbol{R}}\in\mathrm{SO(3)}$. The three columns of the matrix $\,\overline{\boldsymbol{R}}\,$ will be denoted with $\,\boldsymbol{d}_1\,,\,\boldsymbol{d}_2\,,\,\boldsymbol{d}_3\,$. They are usually called the \emph{directors} and can be interpreted as an orthonormal triad of vectors $\{\boldsymbol{d}_i\}$ rigidly attached to each material point, describing thus the microrotations. Using either the direct tensor notation or the matrix notation, we can write respectively
$$\overline{\boldsymbol{R}}=\boldsymbol{d}_i\otimes \boldsymbol{e}_i\qquad\mathrm{or}\qquad
\overline{\boldsymbol{R}}=
\big(\boldsymbol{d}_1\,|\,\boldsymbol{d}_2\,|\,\boldsymbol{d}_3\,\big)_{3\times 3}\,\,.$$
In the reference configuration $\Omega$ the directors attached to every material point $(x_1,x_2,x_3)$ are taken to be $\{\, \boldsymbol{e}_i\,\}$ and after deformation they become $\{\boldsymbol{d}_i(x_1,x_2,x_3)\}$. If we denote by $\,\boldsymbol{\varphi}_0(x_1,x_2,x_3):=x_i\,\boldsymbol{e}_i$ the position vector of points in the reference configuration $\Omega$, then the \emph{displacement vector}  field is $ \boldsymbol{u}=\boldsymbol{\varphi}-\boldsymbol{\varphi}_0\,$.

To introduce the nonlinear strain measures we consider first the so called deformation gradient
\[\boldsymbol{F}=\,\mathrm{Grad}\,\boldsymbol{\varphi}= \boldsymbol{\varphi},_i \otimes \, \boldsymbol{e}_i = \big(\,\boldsymbol{\varphi},_1\,| \, \,\boldsymbol{\varphi},_2\,|\,\,\boldsymbol{\varphi},_3\,\big)_{3\times 3}\,.\]
This motivates to consider the nonsymmetric right stretch tensor (\textbf{the first Cosserat deformation tensor}, see the Cosserat's book \cite{Cosserat09}, p. 123, eq. (43))
\[\overline{\boldsymbol{U}}=\overline{\boldsymbol{R}}^T\boldsymbol{F}=\, (\boldsymbol{e}_i \otimes \boldsymbol{d}_i)(\boldsymbol{\varphi},_j \otimes \boldsymbol{e}_j)= \langle\,\boldsymbol{d}_i\,,\boldsymbol{\varphi},_j\rangle\, \boldsymbol{e}_i \otimes \boldsymbol{e}_j\,.\]
(Remark: we write the above tensors $\overline{\boldsymbol{R}}$ and $\overline{\boldsymbol{U}}$ with superposed  bars in order to distinguish them from the factors  $ {\boldsymbol{R}}$ and $ {\boldsymbol{U}}$ of the classical polar decomposition $\boldsymbol{F}=\boldsymbol{R}\,\boldsymbol{U}$, in which $\boldsymbol{R}$ is orthogonal and $\boldsymbol{U}$ is positive definite, symmetric and which is a standard notation in elasticity.) 
Then, we define the relative Lagrangian \emph{strain measure for stretch} by
\begin{equation}\label{e1}
    \overline{\boldsymbol{E}}\,=\,\overline{\boldsymbol{U}}- \id_3\,=\big(\,\langle\,\boldsymbol{d}_i\,,\boldsymbol{\varphi},_j\rangle-\delta_{ij}\,\big)\, \boldsymbol{e}_i \otimes \boldsymbol{e}_j\,,
\end{equation}
where $\,\delta_{ij} $ is the Kronecker symbol and $\,\id_3:=\boldsymbol{e}_i \otimes \boldsymbol{e}_i$ is the unit tensor in the 3-space.
\begin{remark}
Notice that there are many possibilities to define strain measures for the nonlinear micropolar continuum. For a comparative review of various such definitions we refer to the paper \cite{Pietraszkiewicz09}. The tensor introduced in \eqref{e1} is denoted with $\,\boldsymbol{E}\,$ in  \cite{Pietraszkiewicz09}, but we write it  with a superposed bar $\,\overline{\boldsymbol{E}}\, $ in order to distinguish it from the classical Green-Lagrangian strain tensor $\,\boldsymbol{E}\,=\,\frac{1}{2}\,(\boldsymbol{F}^T\boldsymbol{F}- \id_3)\,$ in three-dimensional finite elasticity.
\end{remark}

For the Lagrangian strain measure for orientation change (curvature) we introduce the \emph{third order curvature tensor}
\begin{equation}\label{e2}
\begin{array}{rcl}
    \boldsymbol{\mathfrak{K}} &\,=\, &\overline{\boldsymbol{R}}^T\,\mathrm{Grad}\, \overline{\boldsymbol{R}}\,=\, \overline{\boldsymbol{R}}^T( \boldsymbol{\overline{\boldsymbol{R}} },_k \otimes \, \boldsymbol{e}_k) \,=\,\big(\overline{\boldsymbol{R}}^T \boldsymbol{\overline{\boldsymbol{R}} },_k \big)\otimes \, \boldsymbol{e}_k \vspace{4pt}\\
     &\,=\, & (\boldsymbol{e}_i \otimes \boldsymbol{d}_i) \big(\boldsymbol{d}_{j,k} \otimes \boldsymbol{e}_j \big)\otimes \, \boldsymbol{e}_k\, =\,
      \langle\,\boldsymbol{d}_i\,,\boldsymbol{d}_{j,k}\,\rangle \, \boldsymbol{e}_i \otimes \boldsymbol{e}_j \otimes\,  \boldsymbol{e}_k\,\,,
     \end{array}
\end{equation}
or in matrix notation
$$ \boldsymbol{\mathfrak{K}} \,=\,  \big(\, \overline{\boldsymbol{R}}^T\overline{\boldsymbol{R}},_1\,| \,\, \overline{\boldsymbol{R}}^T \overline{\boldsymbol{R}},_2\,|\,\,\overline{\boldsymbol{R}}^T \overline{\boldsymbol{R}},_3\,\big)\in \mathbb{R}^{3\times 3\times 3}\,\,.$$
We observe that this tensor corresponds to the \textbf{second Cosserat deformation tensor}, see the Cosserat's book \cite{Cosserat09}, p. 123, eq. (44). Although $\, \boldsymbol{\mathfrak{K}} \,$ is a third order tensor, it has, in fact, only 9 independent components, since $\,\overline{\boldsymbol{R}}^T\overline{\boldsymbol{R}},_k\,$ is skew-symmetric ($k=1,2,3$).
\begin{remark}
With a view towards defining an energy density depending only on the deformation gradient $\boldsymbol{F}$, the microrotation $\overline{\boldsymbol{R}}$ and the derivatives of microrotation $\mathrm{Grad}\,\overline{\boldsymbol{R}}\,\in\R^{3\times 3\times 3}$ we notice that every energy density of the form $W=W\left(\boldsymbol{F},\overline{\boldsymbol{R}},\mathrm{Grad}\,\overline{\boldsymbol{R}}\,\right)$ can be written in the form $W=W\left(\overline{\boldsymbol{U}},\overline{\boldsymbol{K}}\right)$ since by the principle of frame-invariance the energy density has to be left invariant under rigid rotations
\[W(\boldsymbol{Q}\boldsymbol{F},\boldsymbol{Q}\overline{\boldsymbol{R}},\boldsymbol{Q}\,\mathrm{Grad}\,\overline{\boldsymbol{R}})=W(\boldsymbol{F},\overline{\boldsymbol{R}},\mathrm{Grad}\,\overline{\boldsymbol{R}})\qquad\forall\;{\boldsymbol{Q}\in\mathrm{SO}(3)}.\]
and with $\boldsymbol{Q}=\overline{\boldsymbol{R}}^T$ we get (cf. Cosserat's book \cite{Cosserat09}, p. 127)
\[\,W(\boldsymbol{F},\overline{\boldsymbol{R}},\mathrm{Grad}\,\overline{\boldsymbol{R}})=W(\overline{\boldsymbol{R}}^T\boldsymbol{F},\overline{\boldsymbol{R}}^T\overline{\boldsymbol{R}},\overline{\boldsymbol{R}}^T\mathrm{Grad}\,\overline{\boldsymbol{R}})=W(\overline{\boldsymbol{R}}^T\boldsymbol{F},\overline{\boldsymbol{R}}^T\mathrm{Grad}\,\overline{\boldsymbol{R}})=W\left(\overline{\boldsymbol{U}},\boldsymbol{\mathfrak{K}}\right)\,.\]
\end{remark}
If we take the transpose in the last two components of $\,\boldsymbol{\mathfrak{K}} $, then we obtain the third order curvature tensor $\,\boldsymbol{\widetilde{\mathfrak{K}}} \,$ which was used in \cite{Neff_plate04_cmt,Neff_Edinb06,Neff_Forest07}:
\begin{equation}\label{e3}
\begin{array}{rcl}
    \boldsymbol{\widetilde{\mathfrak{K}}} & \,=\, & \boldsymbol{\mathfrak{K}}^{\stackrel{2.3}{T}} \,:=\, \langle\,\boldsymbol{d}_i\,,\boldsymbol{d}_{k,j}\,\rangle \, \boldsymbol{e}_i \otimes \boldsymbol{e}_j \otimes\,  \boldsymbol{e}_k\,=\, (\boldsymbol{e}_i \otimes \boldsymbol{d}_i) \big(\boldsymbol{d}_{k,j} \otimes \boldsymbol{e}_j \big)\otimes \, \boldsymbol{e}_k \vspace{4pt}\\
    & \, =\, &   \overline{\boldsymbol{R}}^T\big(\, \mathrm{Grad}\, \boldsymbol{d }_k \big) \otimes \, \boldsymbol{e}_k \,=\, \big(\, \overline{\boldsymbol{R}}^T\mathrm{Grad}\, \boldsymbol{d }_1\,| \,\, \overline{\boldsymbol{R}}^T \mathrm{Grad}\, \boldsymbol{d }_2\,|\,\,\overline{\boldsymbol{R}}^T \mathrm{Grad}\, \boldsymbol{d }_3\,\big)\,.
 \end{array}
 \end{equation}
In order to avoid working with a third order tensor for the curvature, one can replace $\,\boldsymbol{\mathfrak{K}}\, $ by a second order curvature strain tensor. This can be done in several ways. Indeed, a first way is to consider the axial vector of the skew-symmetric factor $\,\overline{\boldsymbol{R}}^T \boldsymbol{\overline{\boldsymbol{R}} },_k\,$ in the definition \eqref{e2} of $\,\boldsymbol{\mathfrak{K}}\, $ and to introduce thus the second order tensor
\begin{equation}\label{e4}
     \boldsymbol{\Gamma}\,=\, \mathrm{axl}\big(\overline{\boldsymbol{R}}^T \boldsymbol{\overline{\boldsymbol{R}} },_k \big)\otimes \, \boldsymbol{e}_k\,\,.
 \end{equation}
The tensor $\,\boldsymbol{\Gamma}\,$ is a Lagrangian measure for curvature and it is frequently called in the literature the \emph{wryness tensor} (see e.g., \cite{Pietraszkiewicz09}). The relation between $\,\boldsymbol{\Gamma}\,$  and $\,\boldsymbol{\mathfrak{K}}\, $ can be expressed with the help of the third order alternator  tensor
 $$\boldsymbol{\epsilon}=-\id_3\times \id_3=\epsilon_{ijk}\,\boldsymbol{e}_i\otimes\boldsymbol{e}_j\otimes \boldsymbol{e}_k$$
in the form (since $\,\mathrm{axl}\,\boldsymbol{S}=-\frac{1}{2}\,\,\boldsymbol{\epsilon}:\boldsymbol{S}\,$, for any skew-symmetric second order tensor $\boldsymbol{S}$)
\begin{equation}\label{e5}
    \boldsymbol{\Gamma}\,=\,-\,\dfrac{1}{2}\,\,\boldsymbol{\epsilon}:\boldsymbol{\mathfrak{K}} \quad\qquad\mathrm{and}\qquad\quad \boldsymbol{\mathfrak{K}}\,=\, \id_3\times \boldsymbol{\Gamma}\,=\, -\, \boldsymbol{\epsilon}\,\boldsymbol{\Gamma}\,.
\end{equation}
Here, the double dot product `` : '' of two third order tensors $\boldsymbol{A}=A_{ijk}\,\boldsymbol{e}_i\otimes\boldsymbol{e}_j\otimes \boldsymbol{e}_k$ and $\boldsymbol{B}=B_{ijk}\,\boldsymbol{e}_i\otimes\boldsymbol{e}_j\otimes \boldsymbol{e}_k$
is defined as $\,\boldsymbol{A}:\boldsymbol{B}\,=\, A_{irs}B_{rsj}\,\boldsymbol{e}_i\otimes\boldsymbol{e}_j\,$. The cross product ``$\,\times\,$'' of two second order tensors is calculated with help of the rule $\,(\boldsymbol{a}\otimes\boldsymbol{b})\times(\boldsymbol{c}\otimes\boldsymbol{d})= \boldsymbol{a}\otimes(\boldsymbol{b}\times\boldsymbol{c})\otimes\boldsymbol{d}\,$, which holds for any vectors $\boldsymbol{a},\boldsymbol{b},\boldsymbol{c}$ and $\boldsymbol{d}$.

The wryness tensor $\,\boldsymbol{\Gamma}\,$ can be also expressed by means of the directors $\,\boldsymbol{d}_i$ in the form \cite{Tambaca10}
$$\boldsymbol{\Gamma}\,=\, \dfrac{1}{2}\,\,\overline{\boldsymbol{R}}^T\big( \boldsymbol{d}_j\times\boldsymbol{d}_{j,i}  \big) \otimes \boldsymbol{e}_i = \, \dfrac{1}{2}\,\,\epsilon_{iks}\, \langle \boldsymbol{d}_{k,j},\boldsymbol{d}_{s}  \rangle \,\boldsymbol{e}_i\otimes \boldsymbol{e}_j\,\,.$$
For the norm of  $\,\boldsymbol{\Gamma}\,$  we find (since $\|\mathrm{axl}\,\boldsymbol{S}\|^2=\frac{1}{2}\,\,\|\boldsymbol{S}\|^2\,$ for any skew-symmetric  tensor $\boldsymbol{S}$)
\begin{equation}\label{e6}
\begin{array}{rcl}
    \|\,\boldsymbol{\Gamma}\,\|^2 & \,=\, & \displaystyle \sum_{i=1}^3 \| \, \mathrm{axl}\big(\overline{\boldsymbol{R}}^T \boldsymbol{\overline{\boldsymbol{R}} },_i )\, \|^2\,=\, \dfrac{1}{2}\,  \sum_{i=1}^3 \| \, \overline{\boldsymbol{R}}^T \boldsymbol{\overline{\boldsymbol{R}} },_i \|^2\,=\,  \dfrac{1}{2}\,  \|\,
     \boldsymbol{\mathfrak{K}}\,\|^2  \vspace{4pt}\\
    & \, =\, & \dfrac{1}{2}\,  \displaystyle \sum_{i=1}^3 \|\, \overline{\boldsymbol{R}},_i \|^2\,=\,  \dfrac{1}{2}\,  \|  \,  \mathrm{Grad}\, \overline{\boldsymbol{R}}\,\|^2\,\,.
 \end{array}
 \end{equation}

As an alternative to the wryness tensor $\,\boldsymbol{\Gamma}\,$ given by \eqref{e4}, one can proceed as follows: to obtain a second order tensor as curvature measure, one makes use of the $\,\mathrm{Curl}\,$ operator instead of $\,\mathrm{Grad}\,$ in the definition \eqref{e2} of $\,\boldsymbol{\mathfrak{K}}\,$. Thus, we define the \emph{dislocation density tensor} $\,\overline{\boldsymbol{K}}\,$ by
\begin{equation}\label{e7}
    \overline{\boldsymbol{K}}\,=\, \overline{\boldsymbol{R}}^T\,\mathrm{Curl}\, \overline{\boldsymbol{R}}\,=\, -\,\overline{\boldsymbol{R}}^T\big( \boldsymbol{\overline{\boldsymbol{R}} },_i \times \, \boldsymbol{e}_i\big) \,\,.
\end{equation}
Here, the $\,\mathrm{curl}\,$ operator for vector fields $\,\boldsymbol{v}=v_i\, \boldsymbol{e}_i$ has the well-known expression
$$\mathrm{curl}\,\boldsymbol{v}= \epsilon_{ijk}\,v_{j,i}\,\boldsymbol{e}_k= -\boldsymbol{v},_i \times \boldsymbol{e}_i\,\,,$$
while the $\,\mathrm{Curl}\,$ operator for tensor fields $\,\boldsymbol{T}=T_{ij}\,\boldsymbol{e}_i\otimes \boldsymbol{e}_j\,$ is defined as
\begin{equation}\label{e8}
\begin{array}{rcl}
    \mathrm{Curl}\,\boldsymbol{T}  & \,=\, & \epsilon_{jrs}\,T_{is,r}\,\boldsymbol{e}_i\otimes \boldsymbol{e}_j\,=\,  -\,\boldsymbol{T},_i \times\, \boldsymbol{e}_i\,\,,
\end{array}
\end{equation}
or equivalently (since $(\boldsymbol{a}\otimes \boldsymbol{b})\times \boldsymbol{c}=\boldsymbol{a}\otimes (\boldsymbol{b}\times \boldsymbol{c})$)
\begin{equation}
\begin{array}{rcl}
    \mathrm{Curl}\,\boldsymbol{T} & \,=\, & \boldsymbol{e}_i \otimes  \mathrm{curl}\big(\boldsymbol{T}_i\big)\qquad\mathrm{for}\quad \boldsymbol{T}= \boldsymbol{e}_i \otimes  \boldsymbol{T}_i\,\,,
 \end{array}
 \end{equation}
where $\, \boldsymbol{T}_i=T_{ij}\,\boldsymbol{e}_j$ are the 3 rows of the $3\times 3$ matrix $\,\boldsymbol{T}\,$.
In other words, $\,\mathrm{Curl}\,$ is defined row wise as in \cite{Mielke06,Svendsen02}: the rows of the $3\times 3$ matrix $\,\mathrm{Curl}\,\boldsymbol{T}\,$ are respectively the 3 vectors $\,\mathrm{curl}\,\boldsymbol{T}_i\,$ ($i=1,2,3$). Note that some other authors define $\,\mathrm{Curl}\,\boldsymbol{T}\,$ as the transpose of our $\,\mathrm{Curl}\,\boldsymbol{T}\,$ given by \eqref{e8} (see e.g., \cite{Gurtin81,Gurtin00}).

The dislocation density tensor $\,\overline{\boldsymbol{K}}\,$ defined with the help of the $\,\mathrm{Curl}\,$ operator presents some advantages for micropolar and micromorphic media, as it can be seen in \cite{Neff_curl08,GhibaNeffExistence,NeffGhibaMicroModel,MadeoNeffGhibaW,OsterNeff}. We describe next the close relationship between the wryness tensor $\,\boldsymbol{\Gamma}\,$ and the dislocation density tensor $\,\overline{\boldsymbol{K}}\,$. One can prove by a straightforward calculation that the following relations hold \cite{Neff_curl08}
\begin{equation}\label{e9}
    -\,\overline{\boldsymbol{K}}\,=\,\boldsymbol{\Gamma}^T-(\mathrm{tr}\,\boldsymbol{\Gamma})\, \id_3\qquad\quad \mathrm{and}\qquad\quad -\,\boldsymbol{\Gamma}\,=\,\overline{\boldsymbol{K}}^T-\dfrac{1}{2}\, (\mathrm{tr}\,\overline{\boldsymbol{K}})\, \id_3\,\,.
\end{equation}
For infinitesimal strains this formula is well-known under the name \emph{Nye's formula}, see \cite{Neff_curl08} and there $(\,-\,\boldsymbol{\Gamma}\,)$ is also called \emph{Nye's curvature tensor} (cf. \cite{Nye53}). 
We insert next a shortened proof of relations \eqref{e9}: from the definition  \eqref{e7} it follows
\begin{equation}\label{e9,1}
     \overline{\boldsymbol{K}}\,=\,   -\,\big(\overline{\boldsymbol{R}}^T \boldsymbol{\overline{\boldsymbol{R}} },_k\,\big) \times \, \boldsymbol{e}_k \,\,.
\end{equation}
On the other hand, from \eqref{e2} and \eqref{e5}$_2$ we get
\begin{align*}
    \big(\overline{\boldsymbol{R}}^T \boldsymbol{\overline{\boldsymbol{R}} },_k\,\big) \otimes \, \boldsymbol{e}_k  & = \,\boldsymbol{\mathfrak{K}} \,=\, -\, \boldsymbol{\epsilon}\,\boldsymbol{\Gamma}\,=\,-(\epsilon_{ijr}\,\,\boldsymbol{e}_i \otimes\boldsymbol{e}_j\otimes \boldsymbol{e}_r)\,\big(\Gamma_{sk}\, \boldsymbol{e}_s\otimes \boldsymbol{e}_k\big) \vspace{4pt}\\
    & = \,-\, \epsilon_{ijs}\,\Gamma_{sk}\,\,\boldsymbol{e}_i \otimes\boldsymbol{e}_j\otimes \boldsymbol{e}_k\,
\end{align*}
and, consequently,
\begin{equation}\label{e9,2}
    \overline{\boldsymbol{R}}^T \boldsymbol{\overline{\boldsymbol{R}} },_k\,  = \, -\, \epsilon_{ijs}\,\Gamma_{sk}\,\,\boldsymbol{e}_i \otimes\boldsymbol{e}_j\,\,.
\end{equation}
If we replace \eqref{e9,2} in \eqref{e9,1} we obtain
\begin{align*}
    \overline{\boldsymbol{K}}\, & =\,  \big(\epsilon_{ijs}\,\Gamma_{sk}\,\,\boldsymbol{e}_i \otimes\boldsymbol{e}_j\,\big) \times \, \boldsymbol{e}_k\,=\,\epsilon_{ijs}\,\Gamma_{sk}\,\,\boldsymbol{e}_i \otimes \big(\boldsymbol{e}_j \times \, \boldsymbol{e}_k\,\big) \vspace{4pt}\\
    & = \, \epsilon_{ijs}\,\epsilon_{jkm}\, \Gamma_{sk}\,\,\boldsymbol{e}_i \otimes\boldsymbol{e}_m\, = \, \big( \delta_{sk}\,\delta_{im}-\delta_{sm}\,\delta_{ik} \big)\, \Gamma_{sk}\,\boldsymbol{e}_i \otimes\boldsymbol{e}_m \vspace{4pt}\\
    & = \, \Gamma_{ss}\,\boldsymbol{e}_i \otimes\boldsymbol{e}_i\, -  \, \Gamma_{mi}\,\boldsymbol{e}_i \otimes\boldsymbol{e}_m \, = \, (\mathrm{tr}\,\boldsymbol{\Gamma})\, \id_3 \,-\,\boldsymbol{\Gamma}^T\,,
\end{align*}
which is the equation  \eqref{e9}$_1\,$. If we consider the trace and the transpose of \eqref{e9}$_1\,$, then we find \eqref{e9}$_2\,$. Thus, the relations \eqref{e9} are proved.

Taking the norms in relations \eqref{e9} we obtain  the relationships 
\begin{equation}\label{e10}
    \|\,\overline{\boldsymbol{K}}\,\|^2\,=\,\|\,\boldsymbol{\Gamma}\|^2+\big (\mathrm{tr}\,\boldsymbol{\Gamma}\big)^2 \qquad\quad \mathrm{and}\qquad\quad \|\,\boldsymbol{\Gamma}\|^2\,=\,\|\,\overline{\boldsymbol{K}}\|^2-\dfrac{1}{4}\, \big(\mathrm{tr}\,\overline{\boldsymbol{K}}\big)^2\,.
\end{equation}
From \eqref{e9} we deduce the following relations between the traces, symmetric parts and skew-symmetric parts of these two tensors
\begin{equation}\label{e11}
    \mathrm{tr}\,\overline{\boldsymbol{K}}\,=\, 2\,\mathrm{tr}\,\boldsymbol{\Gamma},\qquad \mathrm{skew}\,\overline{\boldsymbol{K}}\,=\, \mathrm{skew}\,\boldsymbol{\Gamma},\qquad \mathrm{dev\,sym}\,\overline{\boldsymbol{K}}\,=\,-\, \mathrm{dev\,sym}\,\boldsymbol{\Gamma},
\end{equation}
where $\,\, \mathrm{dev}\,\boldsymbol{X}=\boldsymbol{X}-\frac{1}{3}\,( \mathrm{tr}\,\boldsymbol{X})\,\id_3\,$ is the deviatoric part of any second order tensor $\boldsymbol{X}$.

In view of \eqref{e9}--\eqref{e11}, we see that one can work either with the wryness tensor $\,\boldsymbol{\Gamma}\,$ or, alternatively, with the dislocation density tensor $\,\overline{\boldsymbol{K}}\,$, since they are in a simple one-to-one relation.

In what follows, we employ the strain and curvature measures $\,\overline{\boldsymbol{E}}=\overline{\boldsymbol{U}}-\id_3\,$ and $\,\overline{\boldsymbol{K}}\,$ to describe the deformation of the Cosserat elastic body.

\section{Constitutive assumptions}\label{sec3}

We begin this section with some simple remarks. Any second order tensor $\,\boldsymbol{X}\in\mathbb{R}^{n\times n}\,$ can be decomposed as direct sum in the form
$$\boldsymbol{X}\,=\, \mathrm{dev\,sym}\, \boldsymbol{X}\, + \, \mathrm{skew}\, \boldsymbol{X}\,   + \, \dfrac{1}{n}\,\big(\mathrm{tr}\,\boldsymbol{X}\big)\,\id_n\,\,,$$
which is the Cartan--Lie--algebra decomposition.
Since the three terms are mutually orthogonal, it follows (for the case $n=3$)
$$\|\,\boldsymbol{X}\|^2\,=\, \|\,\mathrm{dev\,sym}\, \boldsymbol{X}\|^2\, +  \, \|\,\mathrm{skew}\, \boldsymbol{X}\|^2\,   + \, \dfrac{1}{3}\,\big(\mathrm{tr}\,\boldsymbol{X}\big)^2\,\,.$$
This suggests to consider quadratic functions of $\,\overline{\boldsymbol{K}}\,$ of the form
\begin{equation}\label{e11,5}
    B(\overline{\boldsymbol{K}})=\,a_1\,\|\,\mathrm{dev\,sym}\, \overline{\boldsymbol{K}}\|^2\, +  \,a_2 \, \|\,\mathrm{skew}\, \overline{\boldsymbol{K}}\|^2\,   + \, a_3\,\big(\mathrm{tr}\,\overline{\boldsymbol{K}}\big)^2\,,
\end{equation}
where $\,a_i\,$ are some constant coefficients. In fact, every isotropic quadratic form in $\overline{\boldsymbol{K}}$ has this representation. Clearly, $\,B(\overline{\boldsymbol{K}})\,$ is a positive definite quadratic form of $\,\overline{\boldsymbol{K}}\,$ if and only if
\begin{equation}\label{e12}
    a_1>0,\qquad a_2>0\qquad\mathrm{and}\qquad a_3>0.
\end{equation}
In this case there exists a positive constant $\,c_1>0$ such that
\begin{equation}\label{e13}
    a_1\,\|\,\mathrm{dev\,sym}\, \overline{\boldsymbol{K}}\|^2\, +  \,a_2\, \|\,\mathrm{skew}\, \overline{\boldsymbol{K}}\|^2\,   + \, a_3\,\big(\mathrm{tr}\,\overline{\boldsymbol{K}}\big)^2\,\geq\, c_1\, \|\,\overline{\boldsymbol{K}}\|^2\,,
\end{equation}
for any $\,\overline{\boldsymbol{K}}\,$. By virtue of \eqref{e11}, we can rewrite $\,B(\overline{\boldsymbol{K}})\,$ as a function of $\,\boldsymbol{\Gamma}\,$
\begin{equation}\label{e13,5}
    B(\overline{\boldsymbol{K}})\,=\,\widetilde B(\boldsymbol{\Gamma})\,=\,a_1\,\|\,\mathrm{dev\,sym}\, \boldsymbol{\Gamma}\|^2\, +  \,a_2\, \|\,\mathrm{skew}\, \boldsymbol{\Gamma}\|^2\,   + \, 4a_3\,\big(\mathrm{tr}\,\boldsymbol{\Gamma}\big)^2\,.
\end{equation}
Then, the conditions \eqref{e12} ensure the positive definiteness of the quadratic form $\,\widetilde B(\boldsymbol{\Gamma})\,$ as a function of $\, \boldsymbol{\Gamma}\,$.

We present next the main constitutive assumptions. Let $\,W=\widehat{W}(\overline{\boldsymbol{U}},\overline{\boldsymbol{K}})=W(\overline{\boldsymbol{E}},\overline{\boldsymbol{K}})$
be the elastically stored energy density. Assume the following additive split of the energy density
\begin{equation}\label{e14}
    \begin{array}{rcl}
      W(\overline{\boldsymbol{E}},\overline{\boldsymbol{K}}) &=& W_{\mathrm{mp}}(\overline{\boldsymbol{E}} )+ W_{\mathrm{curv}}( \overline{\boldsymbol{K}}),\qquad\mathrm{where} \vspace{4pt}\\
      W_{\mathrm{mp}}(\overline{\boldsymbol{E}} ) &=& \mu\,\|\,\mathrm{dev\,sym}\, \overline{\boldsymbol{E}}\,\|^2\, +  \,\mu_c \, \|\,\mathrm{skew}\, \overline{\boldsymbol{E}}\,\|^2\,   + \, \dfrac{\kappa}{2}\,\big(\mathrm{tr}\,\overline{\boldsymbol{E}}\big)^2\,,\vspace{4pt}\\
       W_{\mathrm{curv}}( \overline{\boldsymbol{K}}) &=& \mu\,L_c^p\,\Big(a_1\|\,\mathrm{dev\,sym}\, \overline{\boldsymbol{K}}\|^2\, +  \,a_2 \|\,\mathrm{skew}\, \overline{\boldsymbol{K}}\|^2\,   + \, a_3\big(\mathrm{tr}\,\overline{\boldsymbol{K}}\big)^2\Big)^{{p}/{2}}\,\,.
    \end{array}
\end{equation}
Here, $\mu$ is the shear modulus and $\kappa$ is the bulk modulus of classical isotropic elasticity, while $\,\mu_c$ is called the \emph{Cosserat couple modulus}. Sine ira et studio we consider here only the standard case 
\begin{equation}\label{e15}
    \mu>0,\qquad \kappa>0,\qquad\text{and}\qquad \mu_c>0\,.
\end{equation}
The parameter $\,L_c$ introduces an internal length which is characteristic for the material and is responsible for size effects; $a_1\,,\, a_2 \,,\, a_3$ are dimensionless constitutive coefficients and $\,p\,$ is a constant exponent present in the curvature energy. We assume
\begin{equation}\label{e16}
     a_1>0,\qquad a_2>0,\qquad a_3>0\qquad \mathrm{and}\qquad L_c>0 , \qquad p\geq 2.
\end{equation}
\begin{remark}\label{rem3}
	In the linearized case 
		\[p=2,\qquad a_1>0,\qquad a_2=a_3=0\]
	is the case with conformal curvature treated in \cite{Neff_Paris_Maugin09}.
\end{remark}
\begin{remark}\label{rem4}
	In this paper we explicitly do not discuss the meaning or physical relevance of the Cosserat couple modulus $\mu_c$. Indeed, in \cite{Neff-PAMM-2004,Neff_Habil04} we have argued that $\mu_c$ should vanish for a continuous body. In this perspective $\mu_c>0$ can be viewed as a discreteness qualifier. Moreover, $\mu_c>0$ is intimately related to the occurrence of boundary layers. \\
	v1710ilThe present existence proof can be modified to include the case $\mu_c=0$ using Korn's extended inequality \cite{Neff-Korn-02} and $p>2$ (at present). Whether $p>2$ is really needed, is an open question \cite{Neff_Habil04, Neff_Forest07}. Note that having $\mu_c=0$ in the linear micropolar-Cosserat model is impossible since then the equations decouple altogether. However, it is possible to consider $\mu_c=0$ in more general linear micromorphic models with meaningful results, we refer the interested reader to \cite{NeffGhibaMicroModel,GhibaNeffExistence,MadeoNeffGhibaW,Neff_Mar14,Madeo14,Neff_Biot07}.
\end{remark}
In view of \eqref{e15} we deduce that $\,W_{\mathrm{mp}}(\overline{\boldsymbol{E}} )$ is positive definite. Thus, the inequality
\begin{equation}\label{e17}
    W_{\mathrm{mp}}(\overline{\boldsymbol{E}} )\,\geq\,c_2\,\|\,\overline{\boldsymbol{E}} \,\|^2
\end{equation}
holds for some constant $c_2>0$ and for any tensor $\overline{\boldsymbol{E}}$.
By virtue of \eqref{e13} and \eqref{e16} we derive that $\,W_{\mathrm{curv}}( \overline{\boldsymbol{K}})$ is also a coercive function, i.e. there exists a constant $c_3>0$ such that
\begin{equation}\label{e18}
    W_{\mathrm{curv}}( \overline{\boldsymbol{K}})\,\geq\,c_3\,\|\,\overline{\boldsymbol{K}} \,\|^p\,.
\end{equation}

In the next section we regard the equilibrium state of the elastic body as a solution of the minimization problem for the energy functional.

\section{Minimization problem and existence theorem}\label{sec4}

Let us introduce the energy functional by
\begin{align}\label{e19}
   I(\boldsymbol{\varphi}, \overline{\boldsymbol{R}}) &= \int_\Omega W(\overline{\boldsymbol{U}},\overline{\boldsymbol{K}})\, \mathrm{d}V -\Pi(\boldsymbol{\varphi}, \overline{\boldsymbol{R}}),\qquad \mathrm{where} \\
   \label{e20}
    \Pi(\boldsymbol{\varphi}, \overline{\boldsymbol{R}}) &= \Pi_f(\boldsymbol{\varphi})+ \Pi_M( \overline{\boldsymbol{R}})+ \Pi_N(\boldsymbol{\varphi})+ \Pi_{M_c}(  \overline{\boldsymbol{R}})
\end{align}
is the \emph{external loading potential}, which consists of the following parts:
\begin{align*}
    \Pi_f(\boldsymbol{\varphi}) &=\int_\Omega \langle\,\boldsymbol{f}\,,\,\boldsymbol{u}\,\rangle\, \mathrm{d}V\quad = \quad \text{ the potential of applied external volume forces }\boldsymbol{f}\,;
    \vspace{3pt}\\
       \Pi_N(\boldsymbol{\varphi}) & = \int_{\Gamma_s} \langle\,\boldsymbol{N}\,,\,\boldsymbol{u}\,\rangle\, \mathrm{d}S\quad = \,\,\, \text{ the potential of applied external surface forces }\boldsymbol{N}\,;
     \vspace{3pt}\\
       \Pi_M( \overline{\boldsymbol{R}}) &=\int_\Omega  \langle\,\boldsymbol{M}\,,\,\overline{\boldsymbol{R}}\,\rangle_{3\times 3} \, \mathrm{d}V;
     \vspace{6pt}\\
     \Pi_{M_c}( \overline{\boldsymbol{R}}) &= \int_{\Gamma_s}  \langle\,\boldsymbol{M}_c\,,\,\overline{\boldsymbol{R}}\,\rangle_{3\times 3} \, \mathrm{d}S;
\end{align*}

Here, $\Gamma_s$ is a nonempty subset of the boundary $\partial \Omega$ such that $\partial \Omega=\Gamma_s\cup \Gamma_d$ (with $\Gamma_s\cap \Gamma_d=\emptyset$). On $\Gamma_s$ we consider traction boundary conditions, while on $\Gamma_d$ we have Dirichlet--type boundary conditions. Moreover $\boldsymbol{M}\in L^q(\Omega,\mathbb{R}^{3\times 3})$ and $\boldsymbol{M}_c\in L^q(\Gamma_s,\mathbb{R}^{3\times 3})$, $\,q\,$ is the conjugated exponent to $p$ ($\frac{1}{p}\,+\frac{1}{q}=1$) and ``$\langle\;\,,\;\rangle_{3\times 3}$'' denotes the scalar product between tensors. 
With abuse of notation, we denote by
\begin{equation}\label{e20,1}
    L^p\big(\Omega,\mathrm{SO}(3)\big):=\big\{\boldsymbol{R}\in L^p(\Omega,\mathbb{R}^{3\times 3})\,|\, \boldsymbol{R}(\boldsymbol{x})\in \mathrm{SO}(3)\,\,\text{for a.e. } \boldsymbol{x}\in\Omega\big\}
\end{equation}
with the induced strong topology of $\, L^p(\Omega,\mathbb{R}^{3\times 3})$. The set $\,L^p\big(\Omega,\mathrm{SO}(3)\big)$ is a closed subset of $\, L^p(\Omega,\mathbb{R}^{3\times 3})$: indeed, let $\big(\boldsymbol{R}^k\big)_{k\geq 1}\subset L^p\big(\Omega,\mathrm{SO}(3)\big)$ be a sequence such that $\,\,\boldsymbol{R}^k\longrightarrow\boldsymbol{R}\,$ in $\,L^p(\Omega,\mathbb{R}^{3\times 3})$. By a known result in functional analysis (see, e.\,g. \cite{Alt06}, p. 56, Lemma 1.20) it follows that on a subsequence $\,\,\boldsymbol{R}^k(\boldsymbol{x}) \longrightarrow\boldsymbol{R}(\boldsymbol{x})\,$ for a.e. $\, \boldsymbol{x}\in\Omega$, and hence, $\, \boldsymbol{R}(\boldsymbol{x})\in \mathrm{SO}(3)\,$ for a.e. $\,\boldsymbol{x}\in\Omega\,$, i.e. $\,\boldsymbol{R}\in L^p\big(\Omega,\mathrm{SO}(3)\big)$.

Similarly,  we introduce the space \cite{Neff_Chelminski_ifb07,Tambaca10}
\begin{equation}\label{e20,2}
    W^{1,p}\big(\Omega,\mathrm{SO}(3)\big):=\big\{\boldsymbol{R}\in W^{1,p}(\Omega,\mathbb{R}^{3\times 3})\,|\, \boldsymbol{R}(\boldsymbol{x})\in \mathrm{SO}(3)\,\,\text{for a.e. } \boldsymbol{x}\in\Omega\big\}
\end{equation}
with the induced strong and weak topologies. We mention that the strong or the weak limit $\,\boldsymbol{R}\,$ of any sequence $\big(\boldsymbol{R}^k\big)\subset W^{1,p}\big(\Omega,\mathrm{SO}(3)\big)$ is also an element of $\,W^{1,p}\big(\Omega,\mathrm{SO}(3)\big)$: indeed, both convergences imply that $\,\,\boldsymbol{R}^k(\boldsymbol{x}) \longrightarrow\boldsymbol{R}(\boldsymbol{x})\,$   a.e. in $\, \Omega\,$, and it follows that $\,\boldsymbol{R}\in W^{1,p}\big(\Omega,\mathrm{SO}(3)\big)$.

The admissible set for the energy functional \eqref{e19} is
\begin{equation}\label{e21}
    \mathcal{A}=\big\{\,(\boldsymbol{\varphi}, \overline{\boldsymbol{R}})\in H^1(\Omega, \mathbb{R}^3)\times W^{1,p}\big(\Omega,\mathrm{SO}(3)\big)\,\big|\,\,\, \boldsymbol{\varphi}_{\big|\Gamma_d}= \boldsymbol{\varphi}_d\,,\,\, \overline{\boldsymbol{R}}_{\big|\Gamma_d}= \boldsymbol{R}_d\,\,\big\},
\end{equation}
where $\,\boldsymbol{\varphi}_d\,$ and $\, \boldsymbol{R}_d\,$ are the given values of the deformation $\,\boldsymbol{\varphi}\,$ and microrotation $\, \overline{\boldsymbol{R}}\,$ on the Dirichlet part $\,\Gamma_d\,$ of the boundary (in the sense of trace).

Taking into account \eqref{e19}--\eqref{e21} we formulate the following minimization problem:\\
\emph{Find a pair $(\boldsymbol{\varphi}, \overline{\boldsymbol{R}})$ minimizing the energy functional $I(\boldsymbol{\varphi}, \overline{\boldsymbol{R}})$ defined by \eqref{e19} over the admissible set $\mathcal{A}$ given by \eqref{e21}.}

\smallskip
The following theorem states the existence of a minimizer to this problem.

\begin{theorem}
Let $\,\Omega\subset\mathbb{R}^3$ be a bounded domain with Lipschitz boundary $\partial\Omega$. The constitutive parameters satisfy \eqref{e15} and \eqref{e16}. Assume that the external load functions and the Dirichlet data $\,\boldsymbol{\varphi}_d\,$, $ \boldsymbol{R}_d\,$ fulfill the regularity requirements
\begin{equation}
     \label{e22}
    \boldsymbol{f}\in L^2(\Omega,\mathbb{R}^{3 }),\quad\boldsymbol{N}\in L^2(\Gamma_s,\mathbb{R}^{3}),\quad
    \boldsymbol{M}\in L^q(\Omega,\mathbb{R}^{3\times 3}),\quad\boldsymbol{M}_c\in L^q(\Gamma_s,\mathbb{R}^{3\times 3})
\end{equation}
\begin{equation}
    \label{e23}
    \boldsymbol{\varphi}_d\in H^1(\Omega, \mathbb{R}^3),\qquad \boldsymbol{R}_d\in W^{1,p}\big(\Omega,\mathrm{SO}(3)\big).
\end{equation}
Then, the above minimization problem \eqref{e19}, \eqref{e21} admits at least one minimizing solution pair $(\boldsymbol{\varphi}, \overline{\boldsymbol{R}})\in\mathcal{A}$.
\end{theorem}
\begin{proof}
We employ the direct methods of the calculus of variations. Let us show first that the external loading potential satisfies the following estimate: there exist some positive constants $c_4\,$, $c_5$ such that
\begin{equation}\label{e24}
    \big|\,\Pi(\boldsymbol{\varphi}, \overline{\boldsymbol{R}})\,\big|\,\leq\,c_4\, \|\boldsymbol{\varphi}\,\|_{H^1(\Omega)}+c_5\,,
\end{equation}
for any $(\boldsymbol{\varphi}, \overline{\boldsymbol{R}})\in\mathcal{A}$. Indeed, since $\,
\boldsymbol{u}:=\boldsymbol{\varphi}-\boldsymbol{\varphi}_d\in H^1(\Omega, \mathbb{R}^3)$ and $\,\boldsymbol{f}\in L^2(\Omega,\mathbb{R}^{3 })$, we obtain from the H\"older inequality
\begin{equation}\label{e25}
    \big|\,\Pi_f(\boldsymbol{\varphi})\,\big|\,=\,\big|\,\int_\Omega \langle\,\boldsymbol{f}\,,\,\boldsymbol{u}\,\rangle\, \mathrm{d}V\,\big|\, \leq \, \|\boldsymbol{f}\,\|_{L^2(\Omega)}\cdot  \|\boldsymbol{u}\,\|_{L^2(\Omega)} \, \leq \, \|\boldsymbol{f}\,\|_{L^2(\Omega)}\cdot  \|\boldsymbol{u}\,\|_{H^1(\Omega)}\,\,.
\end{equation}
Analogously, since $\,
\boldsymbol{u}\in L^2(\Gamma_s, \mathbb{R}^3)$ and $\,\boldsymbol{N}\in L^2(\Gamma_s,\mathbb{R}^{3 })$, we obtain from the H\"older inequality
\begin{equation*}
    \big|\,\Pi_N(\boldsymbol{\varphi})\,\big|\,=\,\big|\,\int_{\Gamma_s} \langle\,\boldsymbol{N}\,,\,\boldsymbol{u}\,\rangle\, \mathrm{d}S\,\big|\, \leq \, \|\boldsymbol{N}\,\|_{L^2(\Gamma_s)}\cdot  \|\boldsymbol{u}\,\|_{L^2(\Gamma_s)}
\end{equation*}
and using the trace theorem $\,\|\boldsymbol{u}\|_{L^2(\Gamma_s)}  \leq\,c\,\|\boldsymbol{u}\|_{H^1(\Omega)}\,$ we find
\begin{equation}\label{e26}
    \big|\,\Pi_N(\boldsymbol{\varphi})\,\big|\, \leq \, c\, \|\boldsymbol{N}\,\|_{L^2(\Gamma_s)}\cdot  \|\boldsymbol{u}\,\|_{H^1(\Omega)} \, .
\end{equation}
On the other hand, for any $\,\overline{\boldsymbol{R}} \in  \mathrm{SO}(3)$ we have $\,\|\overline{\boldsymbol{R}}\,\|^2=3$ and, hence,
$$\|\overline{\boldsymbol{R}}\,\|^2_{L^2(\Omega)}=3\,\mathrm{vol}(\Omega)\qquad\text{and}\qquad \|\overline{\boldsymbol{R}}\,\|^2_{L^2(\Gamma_s)}=3\,\mathrm{area}(\Gamma_s).$$
Since  $\,\Pi_M$ and $ \Pi_{M_c}$ are bounded operators we deduce
\begin{equation}\label{e27}
    \begin{array}{l}
        \big|\, \Pi_M( \overline{\boldsymbol{R}})\, \big|\,\leq \, c_6\,\|\overline{\boldsymbol{R}}\,\|_{L^2(\Omega)}= c_6\big(3\,\mathrm{vol}(\Omega)\,\big)^{1/2}\,,
     \vspace{6pt}\\
     \big|\, \Pi_{M_c}( \overline{\boldsymbol{R}})\,\big|\,\leq\, c_7\, \|\overline{\boldsymbol{R}}\,\|_{L^2(\Gamma_s)}=c_7\,\big(3\,\mathrm{area}(\Gamma_s) \,\big)^{1/2}\,.
    \end{array}
\end{equation}
By virtue of \eqref{e20} and the inequalities \eqref{e25}--\eqref{e27} we see that the estimate \eqref{e24} holds true.\smallskip

We observe that the strain measures satisfy
\begin{equation}\label{e28}
    \overline{\boldsymbol{E}}\in L^2(\Omega,\mathbb{R}^{3\times 3}) ,\qquad  \overline{\boldsymbol{K}}\in L^p(\Omega,\mathbb{R}^{3\times 3}).
\end{equation}
Indeed, since $\,\boldsymbol{\varphi}\in H^1(\Omega, \mathbb{R}^3)$ and $\overline{\boldsymbol{R}}\in \text{SO(3)}$, we find that $\boldsymbol{F}=\text{Grad}\,\boldsymbol{\varphi}\in  L^2(\Omega,\mathbb{R}^{3\times 3}) $ and
\begin{equation*}
    \int_\Omega\|\,\overline{\boldsymbol{U}}\,\|^2\, \mathrm{d}V\,= \int_\Omega\|\,\overline{\boldsymbol{R}}^T\boldsymbol{F}\,\|^2\, \mathrm{d}V\,=  \int_\Omega\|\,\boldsymbol{F}\,\|^2\, \mathrm{d}V\,<+\infty\,.
\end{equation*}
Thus, $\overline{\boldsymbol{E}}=\overline{\boldsymbol{U}}-\id_3\in L^2(\Omega,\mathbb{R}^{3\times 3})$ and the relation \eqref{e28}$_1$ holds. To show \eqref{e28}$_2$ we note that $\overline{\boldsymbol{R}},_{k}\in L^p(\Omega,\mathbb{R}^{3\times 3})$ and from \eqref{e7} it follows
\begin{equation}\label{e29}
    \|\overline{\boldsymbol{K}}\,\|=\,\| \overline{\boldsymbol{R}}^T\,\mathrm{Curl}\, \overline{\boldsymbol{R}}\,\|\,=\, \| \mathrm{Curl}\, \overline{\boldsymbol{R}}\,\|\,=\,\|  \boldsymbol{\overline{\boldsymbol{R}} },_k \times \, \boldsymbol{e}_k\,\|\,.
\end{equation}
For any second order tensor $\,\boldsymbol{S}=S_{ij}\,\boldsymbol{e}_i\otimes\boldsymbol{e}_j\,$ we have
\begin{equation*}
    \begin{array}{c}
      \| \boldsymbol{S}\times\boldsymbol{e}_1\|^2=\|S_{ij}\, \boldsymbol{e}_i\otimes(\boldsymbol{e}_j\times\boldsymbol{e}_1)\|^2= \| S_{i3}\,\boldsymbol{e}_i\otimes\boldsymbol{e}_2- S_{i2}\,\boldsymbol{e}_i\otimes\boldsymbol{e}_3\|^2 \\
    =\displaystyle \sum_{i=1}^3 \big[(S_{i3})^2+(S_{i2})^2\big]\,\leq\,S_{ij}\,S_{ij} =\|\boldsymbol{S}\,\|^2\,,
    \end{array}
\end{equation*}
and analogously for $\boldsymbol{e}_2$ and $\boldsymbol{e}_3\,$. We see that
$$\|  \boldsymbol{\overline{\boldsymbol{R}} },_k \times \, \boldsymbol{e}_k\,\|\,\leq \|  \boldsymbol{\overline{\boldsymbol{R}} },_k  \,\|\,\qquad\text{for any}\quad k=1,2,3\,\,\,(\text{not summed})$$
and then from \eqref{e29} we get
\begin{equation*}
     \int_\Omega\|\,\overline{\boldsymbol{K}}\,\|^p\, \mathrm{d}V\,= \int_\Omega\|\,\boldsymbol{\overline{\boldsymbol{R}} },_k \times \, \boldsymbol{e}_k\,\|^p\, \mathrm{d}V\, \leq \int_\Omega
     \displaystyle \sum_{k=1}^3
     \|\,\boldsymbol{\overline{\boldsymbol{R}} },_k\,\|^p\, \mathrm{d}V\,< \,+\infty,
\end{equation*}
so that \eqref{e28}$_2$ holds.\smallskip

Next we prove the existence of some positive constants $k_1\,,\,k_2\,,\,k_3\,$ such that
\begin{equation}\label{e30}
     I(\boldsymbol{\varphi}, \overline{\boldsymbol{R}})\,\geq \, k_1\,\|\boldsymbol{\varphi}-\boldsymbol{\varphi}_d\,\|^2_{H^1(\Omega)}  - k_2\,\|\boldsymbol{\varphi}-\boldsymbol{\varphi}_d\,\|_{H^1(\Omega)} - k_3\,,\qquad \forall\, (\boldsymbol{\varphi}, \overline{\boldsymbol{R}})\in\mathcal{A}.
\end{equation}
On the basis of \eqref{e14}$_1\,$, \eqref{e19}, \eqref{e17} and \eqref{e24}, we have
\begin{equation}\label{e31}
     I(\boldsymbol{\varphi}, \overline{\boldsymbol{R}})\,\geq\, \int_\Omega W_{\mathrm{mp}}(\overline{\boldsymbol{E}} )\, \mathrm{d}V -\Pi(\boldsymbol{\varphi}, \overline{\boldsymbol{R}})\,\geq\, c_2\int_\Omega \|\,\overline{\boldsymbol{E}} \,\|^2 \, \mathrm{d}V -c_4 \,\|\,\boldsymbol{\varphi} \,\|_{H^1(\Omega)}-c_5\,\,.
\end{equation}
On the other hand, we can write
\begin{equation*}
    \begin{array}{rcl}
      \|\,\overline{\boldsymbol{E}} \,\|^2 & = & \|\,\overline{\boldsymbol{U}} -\id_3\,\|^2=\|\,\overline{\boldsymbol{R}}^T\boldsymbol{F}-\id_3 \,\|^2 = \|\,\boldsymbol{F}-\overline{\boldsymbol{R}}\,\|^2=\|(\boldsymbol{\varphi},_i- \boldsymbol{d}_i)\otimes \boldsymbol{e}_i\|^2 \vspace{3pt}\\
      & = &\langle\,(\boldsymbol{\varphi},_i- \boldsymbol{d}_i)\,,(\boldsymbol{\varphi},_i- \boldsymbol{d}_i)\,\rangle= \langle\,\boldsymbol{\varphi},_i\,,\,\boldsymbol{\varphi},_i\rangle- 2\, \langle\, \boldsymbol{\varphi},_i\,,\, \boldsymbol{d}_i\,\rangle+3
    \end{array}
\end{equation*}
and using the Cauchy--Schwarz inequality we get
\begin{equation*}
    \int_\Omega \|\,\overline{\boldsymbol{E}} \,\|^2\, \mathrm{d}V\,  \geq  \int_\Omega  \langle\,\boldsymbol{\varphi},_i\,,\,\boldsymbol{\varphi},_i\rangle\, \mathrm{d}V\,- 2\Big(\int_\Omega  \langle\,\boldsymbol{d}_i\,,\,\boldsymbol{d}_i\,\rangle\, \mathrm{d}V\Big)^{1/2} \Big(\int_\Omega  \langle\,\boldsymbol{\varphi},_j\,,\,\boldsymbol{\varphi},_j\rangle\, \mathrm{d}V\Big)^{1/2} +3\,\mathrm{vol}(\Omega)
\end{equation*}
and consequently
\begin{equation}\label{e32}
    \|\,\overline{\boldsymbol{E}} \,\|^2_{L^2(\Omega)}\,\geq\,  \|\,\text{Grad}\,\boldsymbol{\varphi} \,\|^2_{L^2(\Omega)} - 2\,\big(3\,\mathrm{vol}(\Omega)\big)^{1/2}\, \|\,\boldsymbol{\varphi} \,\|_{H^1(\Omega)}+ 3\,\mathrm{vol}(\Omega).
\end{equation}
For any $(\boldsymbol{\varphi}, \overline{\boldsymbol{R}})\in\mathcal{A}$ we have $\,\boldsymbol{\varphi}-\boldsymbol{\varphi}_d\in H^1(\Omega, \mathbb{R}^3)$ and $\,\boldsymbol{\varphi}-\boldsymbol{\varphi}_d=\,\boldsymbol{0}$ on $\Gamma_d\,$. Applying the Poincar\'e inequality we find
$$ \|\text{Grad}\big(\boldsymbol{\varphi}-\boldsymbol{\varphi}_d\big) \,\|^2_{L^2(\Omega)} \,\geq\,c_p^+\,\|\,\boldsymbol{\varphi}-\boldsymbol{\varphi}_d \,\|^2_{H^1(\Omega)}\,\,,$$
which implies
\begin{equation}\label{e33}
    \|\text{Grad}\boldsymbol{\varphi}  \,\|^2_{L^2(\Omega)} \,\geq\,c_p^+\,\|\,\boldsymbol{\varphi}-\boldsymbol{\varphi}_d \,\|^2_{H^1(\Omega)} - 2\,
     \|\text{Grad}\boldsymbol{\varphi}_d  \,\|_{L^2(\Omega)}\, \|\boldsymbol{\varphi} \|_{H^1(\Omega)}- \|\text{Grad}\boldsymbol{\varphi}_d  \,\|^2_{L^2(\Omega)}\,.
\end{equation}
If we insert \eqref{e33} into \eqref{e32} and the result into \eqref{e31}, we obtain that
$$ I(\boldsymbol{\varphi}, \overline{\boldsymbol{R}})\,\geq \, K_1\,\|\boldsymbol{\varphi}-\boldsymbol{\varphi}_d\|^2_{H^1(\Omega)}  - K_2\,\|\boldsymbol{\varphi}\,\|_{H^1(\Omega)} - K_3\,,$$
for some constants $K_1\,,\,K_2\,,\,K_3\,>0$. Hence, the inequality \eqref{e30} is valid.
From \eqref{e30} we can infer that the functional $\,I(\boldsymbol{\varphi}, \overline{\boldsymbol{R}}\,)\,$ is bounded from below on the admissible set $\mathcal{A}$.
\smallskip

In view of \eqref{e21}, \eqref{e23} we have $\,(\boldsymbol{\varphi}_d, \boldsymbol{R}_d)\in\mathcal{A}$, and from  \eqref{e24}, \eqref{e28} we deduce
\begin{equation*}
     I(\boldsymbol{\varphi}_d, \boldsymbol{R}_d) \,\leq\, \int_\Omega W_{\mathrm{mp}}(\overline{\boldsymbol{E}}_d )+W_{\mathrm{curv}}(\overline{\boldsymbol{K}}_d )\, \mathrm{d}V +c_4 \,\|\,\boldsymbol{\varphi}_d \,\|_{H^1(\Omega)}+c_5 < +\infty
\end{equation*}
where $\,\overline{\boldsymbol{E}}_d=\boldsymbol{R}_d^T\,\text{Grad}\,\boldsymbol{\varphi}_d- \id_3\in L^2(\Omega,\mathbb{R}^{3\times 3})\,$ and
$\, \overline{\boldsymbol{K}}_d= \boldsymbol{R}_d^T\,\mathrm{Curl}\, \boldsymbol{R}_d\in L^p(\Omega,\mathbb{R}^{3\times 3})$.

Thus, we can choose an infimizing sequence  $\,(\boldsymbol{\varphi}^k, \overline{\boldsymbol{R}}^k)_{k\geq 1}\in\mathcal{A}\,$ such that
\begin{equation}\label{e34}
    \displaystyle \lim_{k\rightarrow\infty} I(\boldsymbol{\varphi}^k, \overline{\boldsymbol{R}}^k)= \inf_{(\boldsymbol{\varphi}, \overline{\boldsymbol{R}})\in\mathcal{A}} I(\boldsymbol{\varphi}, \overline{\boldsymbol{R}}\,)\qquad\text{and}\qquad I(\boldsymbol{\varphi}^k, \overline{\boldsymbol{R}}^k)\leq I(\boldsymbol{\varphi}_d, \boldsymbol{R}_d\,) < +\infty,
\end{equation}
for any $\,k\geq 1$. Taking into account \eqref{e30} and \eqref{e34}$_2$ we find
$$ I(\boldsymbol{\varphi}_d, \boldsymbol{R}_d\,)\,\geq\, I(\boldsymbol{\varphi}^k, \overline{\boldsymbol{R}}^k) \,\geq \, k_1\,\|\boldsymbol{\varphi}^k-\boldsymbol{\varphi}_d\,\|^2_{H^1(\Omega)}  - k_2\,\|\boldsymbol{\varphi}^k-\boldsymbol{\varphi}_d\,\|_{H^1(\Omega)} - k_3\,.$$
Since $k_1>0$, we deduce that $\,\|\boldsymbol{\varphi}^k-\boldsymbol{\varphi}_d\,\|_{H^1(\Omega)}$ is bounded, i.e. the sequence $\,(\boldsymbol{\varphi}^k)_{k\geq 1}$ is bounded in $H^1(\Omega,\mathbb{R}^{3})$.
Then, there exists an element $\,\widehat{\boldsymbol{\varphi}}\in H^1(\Omega,\mathbb{R}^{3})$ and a subsequence of $\,(\boldsymbol{\varphi}^k)_{k\geq 1}\,$, not relabeled, such that $\,\boldsymbol{\varphi}^k $ converges weakly to  $\,\widehat{\boldsymbol{\varphi}}\,$ in $ H^1(\Omega,\mathbb{R}^{3})$ and $\,\boldsymbol{\varphi}^k $ converges strongly to  $\,\widehat{\boldsymbol{\varphi}}\,$ in $ L^2(\Omega,\mathbb{R}^{3})$, i.e.
\begin{equation}\label{e35}
    \boldsymbol{\varphi}^k \stackrel{H^1(\Omega) }{-\!\!\!\rightharpoonup } \widehat{\boldsymbol{\varphi}}\qquad\text{and}\qquad  \boldsymbol{\varphi}^k \stackrel{L^2 (\Omega)}{\longrightarrow } \widehat{\boldsymbol{\varphi}}\,\,.
\end{equation}
Let us denote the strain measures corresponding to $\,(\boldsymbol{\varphi}^k, \overline{\boldsymbol{R}}^k)$ by
\begin{equation}\label{e36}
\begin{array}{l}
    \overline{\boldsymbol{E}}^k\, = \,\overline{\boldsymbol{U}}^k- \id_3=\big(\overline{\boldsymbol{R}}^k\big)^T\,\text{Grad}\, \boldsymbol{\varphi}^k- \id_3\in L^2(\Omega,\mathbb{R}^{3\times 3}),
    \vspace{4pt}\\
    \overline{\boldsymbol{K}}^k=\big(\overline{\boldsymbol{R}}^k\big)^T\,\mathrm{Curl}\, \overline{\boldsymbol{R}}^k\in L^p(\Omega,\mathbb{R}^{3\times 3}).
    \end{array}
\end{equation}
In view of \eqref{e19},\eqref{e24} and \eqref{e34}$_2\,$, we have
\begin{equation*}
    \int_\Omega W_{\mathrm{mp}}(\overline{\boldsymbol{E}}^k )\, \mathrm{d}V \, \leq\,  I(\boldsymbol{\varphi}^k, \overline{\boldsymbol{R}}^k)+ \Pi (\boldsymbol{\varphi}^k, \overline{\boldsymbol{R}}^k) \,\leq\, I(\boldsymbol{\varphi}_d, \boldsymbol{R}_d) +c_4 \,\|\,\boldsymbol{\varphi}^k \,\|_{H^1(\Omega)}+c_5
\end{equation*}
and from \eqref{e17} and the boundedness of $\, \boldsymbol{\varphi}^k$ in $H^1(\Omega;\mathbb{R}^3)$ we infer the existence of a constant $M_1>0$ such that
$$\|\overline{\boldsymbol{E}}^k\|_{L^2(\Omega)}\,\leq\,M_1\,\,.$$
Hence, there exists an element $\,\widehat{\boldsymbol{E}}\in L^2(\Omega,\mathbb{R}^{3\times 3})\,$ and a subsequence of $(\overline{\boldsymbol{E}}^k)$, not relabeled, with
\begin{equation}\label{e37}
     \overline{\boldsymbol{E}}^k \stackrel{L^2(\Omega) }{-\!\!\!\rightharpoonup } \widehat{\boldsymbol{E}}\,.
\end{equation}
Similarly, from \eqref{e19},\eqref{e24} and \eqref{e34}$_2\,$, we find
\begin{equation*}
    \int_\Omega W_{\mathrm{curv}}(\overline{\boldsymbol{K}}^k )\, \mathrm{d}V \, \leq\,  I(\boldsymbol{\varphi}_d, \boldsymbol{R}_d) +c_4 \,\|\,\boldsymbol{\varphi}^k \,\|_{H^1(\Omega)}+c_5
\end{equation*}
and using \eqref{e18} we deduce that $(\overline{\boldsymbol{K}}^k)$ is bounded in $\,L^p(\Omega,\mathbb{R}^{3\times 3})\,$:
\begin{equation}\label{e38}
\|\overline{\boldsymbol{K}}^k\|_{L^p(\Omega)}\,\leq\,M_2\qquad\text{for some constant } M_2\,.
\end{equation}
This implies the existence of a  subsequence (not relabeled) of $(\overline{\boldsymbol{K}}^k)$ which converges weakly to a limit $\,\widehat{\boldsymbol{K}}\in L^p(\Omega,\mathbb{R}^{3\times 3})\,$:
\begin{equation}\label{e39}
     \overline{\boldsymbol{K}}^k \stackrel{L^p(\Omega) }{-\!\!\!\rightharpoonup } \widehat{\boldsymbol{K}}\,.
\end{equation}
Let us show that the  sequence  $\,( \overline{\boldsymbol{R}}^k)_{k\geq 1}$ also admits a convergent subsequence. Using relations of the type \eqref{e6} and \eqref{e10} we derive
$$\|\overline{\boldsymbol{K}}\,\|^2\,\geq\,\|\boldsymbol{\Gamma}\,\|^2\,=\,\dfrac{1}{2}\,\|\,\text{Grad}\, \overline{\boldsymbol{R}}\,\|^2\,\geq\,\dfrac{1}{2}\,\|\,  \overline{\boldsymbol{R}},_i\,\|^2\qquad\text{for}\quad i=1,2,3.$$
If we write this inequality for $\,\overline{\boldsymbol{K}}^k$, $\overline{\boldsymbol{R}}^k$ and integrate over $\,\Omega\,$ we get
\begin{equation*}
    \int_\Omega \|\,\overline{\boldsymbol{K}}^k\, \|^p\, \mathrm{d}V \, \geq\,  \Big(\dfrac{1}{2}\Big)^{p/2}\int_\Omega \|\,\overline{\boldsymbol{R}}_{,i}^k \,\|^p\, \mathrm{d}V,\qquad\text{i.e.}
\end{equation*}
\begin{equation}\label{e40}
 \|\,\overline{\boldsymbol{R}}_{,i}^k \,\|_{L^p(\Omega)}\,\leq\, \sqrt{2}\,\, \|\,\overline{\boldsymbol{K}}^k \,\|_{L^p(\Omega)}\qquad\text{for}\quad i=1,2,3.
\end{equation}
From \eqref{e38} and \eqref{e40} we deduce that $\big(\overline{\boldsymbol{R}}_{,i}^k\big)_{k\geq 1}$ is bounded in $\,L^p(\Omega,\mathbb{R}^{3\times 3})\,$ and since $\,\|\,\overline{\boldsymbol{R}}^k\, \|^2=3\,$, it follows that $\big(\overline{\boldsymbol{R}}_{,i}^k\big)_{k\geq 1}$ is a bounded sequence in $\,W^{1,p}(\Omega,\mathbb{R}^{3\times 3})$. Consequently, there exists an element $\widehat{\boldsymbol{R}}\in W^{1,p}(\Omega,\mathbb{R}^{3\times 3})$ such that (on a subsequence, not relabeled)
\begin{equation}\label{e41}
    \overline{\boldsymbol{R}}^k \stackrel{W^{1,p}(\Omega) }{-\!\!\!\rightharpoonup } \widehat{\boldsymbol{R}}\qquad\text{and}\qquad  \overline{\boldsymbol{R}}^k \stackrel{L^p(\Omega) }{\longrightarrow } \widehat{\boldsymbol{R}}\,\,.
\end{equation}
In view of the mentioned properties of the spaces defined in \eqref{e20,1} and \eqref{e20,2}, we have moreover that $\,\widehat{\boldsymbol{R}}\in W^{1,p}\big(\Omega,\mathrm{SO}(3)\big)$.
\smallskip

As the next step, we want to show that the limits
$\,\widehat{\boldsymbol{\varphi}}$, $\widehat{\boldsymbol{R}}$, $\widehat{\boldsymbol{E}}$ and $\,\widehat{\boldsymbol{K}}$ defined by   \eqref{e35}, \eqref{e37}, \eqref{e39} and \eqref{e41}  are connected through the relations
\begin{equation}\label{e44}
    \widehat{\boldsymbol{E}}\, = \widehat{\boldsymbol{R}}{}^T\,\text{Grad}\, \widehat{\boldsymbol{\varphi}}- \id_3\,\,,
    \qquad
    \widehat{\boldsymbol{K}}=\widehat{\boldsymbol{R}}{}^T\,\mathrm{Curl}\, \widehat{\boldsymbol{R}}\,.
\end{equation}
We shall use the following known result about the (weak) convergence in Lebesgue spaces (see, e.g. \cite{Dacorogna04}, p. 24), written in the tensorial case:
\medskip

\textbf{Proposition.} \emph{Let $\,p\in[1,+\infty)$, let $\,q\,$ be the conjugate exponent to $\,p\,$ $\big(\frac{1}{p}\, +\,\frac{1}{q}\,=1\big)$ and let the sequences $\,(\boldsymbol{U}_k)_{k\geq 1}\subset L^p(\Omega,\mathbb{R}^{3\times 3})\,$, $\,(\boldsymbol{V}_k)_{k\geq 1}\subset L^q(\Omega,\mathbb{R}^{3\times 3})\,$ such that
$$ \boldsymbol{U}_k \stackrel{L^p(\Omega) }{-\!\!\!\rightharpoonup } \boldsymbol{U}\qquad\text{and}\qquad  \boldsymbol{V}_k \stackrel{L^q(\Omega) }{\longrightarrow } \boldsymbol{V}\,\,.$$
Then, it follows $\,\,\boldsymbol{V}_k\boldsymbol{U}_k \stackrel{L^1(\Omega) }{-\!\!\!\rightharpoonup } \boldsymbol{V}\boldsymbol{U}\,$.}\medskip

To prove \eqref{e44}, let us consider the sequence $\big( \overline{\boldsymbol{R}}^{k,T}\,\text{Grad}\, \boldsymbol{\varphi}^k\big)_{k\geq 1}\subset
L^2(\Omega,\mathbb{R}^{3\times 3})\,$. In view of \eqref{e35}$_1$ we have $ \, \text{Grad}\, \boldsymbol{\varphi}^k   \stackrel{L^2(\Omega) }{-\!\!\!\rightharpoonup }    \text{Grad}\,\widehat{ \boldsymbol{\varphi}}\,$. Further, from \eqref{e41}$_2$ and $\,p\geq 2\,$ we deduce that $\, \overline{\boldsymbol{R}}^k \stackrel{L^2(\Omega) }{\longrightarrow } \widehat{\boldsymbol{R}}\,$ and using the above Proposition we find
\begin{equation}\label{e45}
    \overline{\boldsymbol{R}}^{k,T}\,\text{Grad}\, \boldsymbol{\varphi}^k  \,\stackrel{L^1(\Omega) }{-\!\!\!\rightharpoonup } \, \widehat{\boldsymbol{R}}{}^T\,  \text{Grad}\,\widehat{ \boldsymbol{\varphi}}\,\,,\qquad\text{i.e.}\qquad  \overline{\boldsymbol{E}}^k \,\stackrel{L^1(\Omega) }{-\!\!\!\rightharpoonup }\, \widehat{\boldsymbol{E}}\,.
\end{equation}
Comparing the last relation with \eqref{e37} we see that \eqref{e44}$_1$ holds true. To show \eqref{e44}$_2$ we observe that $\,\overline{\boldsymbol{R}}_{,i}^k \stackrel{L^p(\Omega) }{-\!\!\!\rightharpoonup }
\widehat{\boldsymbol{R}}_{,i}\,$ for $i=1,2,3$, by virtue of \eqref{e41}$_1\,$. Then, we have \begin{equation}\label{e46}
    \overline{\boldsymbol{R}}_{,i}^k\times \boldsymbol{e}_i \stackrel{L^p(\Omega) }{-\!\!\!\rightharpoonup }
\widehat{\boldsymbol{R}}_{,i} \times \boldsymbol{e}_i\,\,,\qquad\text{i.e.}\qquad \text{Curl}\,\overline{\boldsymbol{R}}^k \stackrel{L^p(\Omega) }{-\!\!\!\rightharpoonup }
\text{Curl}\,\widehat{\boldsymbol{R}}\,.
\end{equation}
On the other hand, from \eqref{e41}$_2$ and $\,q=\dfrac{p}{p-1}\,\leq p\,$ it follows that
$\, \overline{\boldsymbol{R}}^k \stackrel{L^q(\Omega) }{\longrightarrow } \widehat{\boldsymbol{R}}$. We can apply now the above Proposition to obtain
\begin{equation}\label{e47}
    \overline{\boldsymbol{R}}^{k,T}\,\text{Curl}\,\overline{\boldsymbol{R}}^k \,\stackrel{L^1(\Omega) }{-\!\!\!\rightharpoonup }  \, \widehat{\boldsymbol{R}}{}^T\,
\text{Curl}\,\widehat{\boldsymbol{R}} \,,\qquad\text{i.e.}\qquad   \overline{\boldsymbol{K}}^k \, \stackrel{L^1(\Omega) }{-\!\!\!\rightharpoonup } \,\,  \widehat{\boldsymbol{R}}{}^T\,
\text{Curl}\,\widehat{\boldsymbol{R}}\,.
\end{equation}
By comparison of \eqref{e39} and \eqref{e47}$_2$ we deduce that $\,\widehat{\boldsymbol{K}}=\widehat{\boldsymbol{R}}{}^T\,
\text{Curl}\,\widehat{\boldsymbol{R}}\,$ and thus the relations \eqref{e44} are proved. \smallskip

We want to show next that $\,(\widehat{ \boldsymbol{\varphi}}, \widehat{\boldsymbol{R}})\,$ is a minimizer of the energy functional $I$. To this aim, we prove first that
\begin{equation}\label{e48}
    W(\overline{\boldsymbol{E}},\overline{\boldsymbol{K}}) \quad\text{is convex in}\quad(\overline{\boldsymbol{E}},\overline{\boldsymbol{K}}).
\end{equation}
Indeed, $\,W_{\mathrm{mp}}(\overline{\boldsymbol{E}} )$ is convex in $\, \overline{\boldsymbol{E}} \,$, since $\,W_{\mathrm{mp}}(\overline{\boldsymbol{E}} )$ is a quadratic form in $\,E_{ij}$   and the Hessian matrix is positive definite, cf. \eqref{e14}$_2$ and \eqref{e17}. Analogously, the function $\,f(\overline{\boldsymbol{K}})\,$ defined by \eqref{e11,5} is convex in $\,\overline{\boldsymbol{K}}$, since it is quadratic and its Hessian matrix is positive definite (by virtue of $\,a_1 , a_2, a_3>0$). We can write the curvature energy density $\, W_{\mathrm{curv}}( \overline{\boldsymbol{K}})$ in \eqref{e14}$_3$ as the composition
$$ W_{\mathrm{curv}}( \overline{\boldsymbol{K}})\,=\,\mu\,L_c^p\,\,g\big(B(\overline{\boldsymbol{K}})\big),\qquad\text{
with }\quad g(t):=t^{p/2},\quad g:[0,+\infty)\rightarrow\mathbb{R}.$$
Since $\,g\,$ is a convex and increasing function and $B$ is convex, we see that $\, W_{\mathrm{curv}}( \overline{\boldsymbol{K}})\,$  is also a convex function of $\,\overline{\boldsymbol{K}}\,$. In view of the decomposition \eqref{e14}$_1\,$, the assertion \eqref{e48} holds true.

By virtue of the convergences \eqref{e37}, \eqref{e39} and the convexity property \eqref{e48} we obtain
\begin{equation}\label{e49}
    \int_\Omega W (\widehat{\boldsymbol{E}},\widehat{\boldsymbol{K}} )\, \mathrm{d}V \, \leq\,
    \displaystyle\liminf_{k\rightarrow\infty}
    \int_\Omega W (\overline{\boldsymbol{E}}^k,\overline{\boldsymbol{K}}^k )\, \mathrm{d}V\,.
\end{equation}
For the potential of external loads we notice that
\begin{equation}\label{e50}
    \displaystyle\lim_{k\rightarrow\infty} \Pi_f(\boldsymbol{\varphi}^k)= \Pi_f(\widehat{\boldsymbol{\varphi}}),\qquad
     \displaystyle\lim_{k\rightarrow\infty}  \Pi_N(\boldsymbol{\varphi}^k)= \Pi_N (\widehat{\boldsymbol{\varphi}}),
\end{equation}
\begin{equation}\label{e51}
      \displaystyle\lim_{k\rightarrow\infty} \Pi_M( \overline{\boldsymbol{R}}^k)= \Pi_M( \widehat{\boldsymbol{R}}),\qquad
     \displaystyle\lim_{k\rightarrow\infty}  \Pi_{M_c}( \overline{\boldsymbol{R}}^k)= \Pi_{M_c}( \widehat{\boldsymbol{R}}).
\end{equation}
Indeed, the relations \eqref{e50} are a consequence of the assumptions \eqref{e22} and the convergence \eqref{e35}. Similarly, the relations \eqref{e51} follow from the convergence \eqref{e41} and the assumed continuity of the operators $\,\Pi_M$ and $  \Pi_{M_c}\,$. The relations \eqref{e50} and \eqref{e51} show that
$$\displaystyle\lim_{k\rightarrow\infty} \Pi(\boldsymbol{\varphi}^k, \overline{\boldsymbol{R}}^k)=\, \Pi(\widehat{\boldsymbol{\varphi}}, \widehat{\boldsymbol{R}})$$
and from \eqref{e19}, \eqref{e49} we deduce
\begin{equation}\label{e52}
    I(\widehat{\boldsymbol{\varphi}}, \widehat{\boldsymbol{R}}) \,\, \leq\,  \, \displaystyle\liminf_{k\rightarrow\infty}\, I(\boldsymbol{\varphi}^k, \overline{\boldsymbol{R}}^k)\,.
\end{equation}
Taking into account \eqref{e34} and \eqref{e52} we find
\begin{equation}\label{e53}
    I(\widehat{\boldsymbol{\varphi}}, \widehat{\boldsymbol{R}}) \,\, \leq\,  \, \displaystyle\inf_{(\boldsymbol{\varphi}, \overline{\boldsymbol{R}})\in\mathcal{A}}\, I(\boldsymbol{\varphi}, \overline{\boldsymbol{R}})\,.
\end{equation}
Finally, we verify that $\,(\widehat{\boldsymbol{\varphi}}, \widehat{\boldsymbol{R}}) \in\mathcal{A}\,$ by means of the definition  \eqref{e21}. We have already seen in \eqref{e35}, \eqref{e41} that $\,(\widehat{\boldsymbol{\varphi}}, \widehat{\boldsymbol{R}}) \in H^1(\Omega,\mathbb{R}^{3})\times W^{1,p}\big(\Omega,\mathrm{SO}(3)\big)$. In order to check the Dirichlet boundary conditions, we use the embedding theorem on the boundary (see, e.g. \cite{Alt06}, Theorem A6.13, p. 272):
$$\text{if }\quad u_k \stackrel{W^{1,p}(\Omega) }{-\!\!\!\rightharpoonup } u\,,\qquad\text{then }\quad u_k \stackrel{L^{p}(\partial\Omega) }{\longrightarrow } u\,,$$
written in the vectorial or tensorial case.

Thus, from \eqref{e35}$_1$ we see that $\,\boldsymbol{\varphi}^k \stackrel{L^2(\Gamma_d) }{\longrightarrow } \widehat{\boldsymbol{\varphi}}\,$. But we have $\,{\boldsymbol{\varphi}^k}_{\big|\Gamma_d}= \boldsymbol{\varphi}_d\,$ for any $\,k\geq 1\,$ and hence,  $\,\widehat{\boldsymbol{\varphi}}_{\big|\Gamma_d}= \boldsymbol{\varphi}_d\,$ must hold. Analogously, from \eqref{e41}$_1$ it follows that $\, \overline{\boldsymbol{R}}^k \stackrel{L^p(\Gamma_d) }{\longrightarrow } \widehat{\boldsymbol{R}}\,$ and using
${\overline{\boldsymbol{R}}^k}_{\big|\Gamma_d}= \boldsymbol{R}_d\,$
we deduce that
$\widehat{\boldsymbol{R}}_{\big|\Gamma_d}= \boldsymbol{R}_d\,$.

Therefore, $\,(\widehat{\boldsymbol{\varphi}}, \widehat{\boldsymbol{R}})\,$ is indeed an element of the admissible set $\,\mathcal{A}\,$ and the relation \eqref{e53} can be rewritten as
\begin{equation}
     I(\widehat{\boldsymbol{\varphi}}, \widehat{\boldsymbol{R}}) \,\, =\,  \, \displaystyle\min_{(\boldsymbol{\varphi}, \overline{\boldsymbol{R}})\in\mathcal{A}}\, I(\boldsymbol{\varphi}, \overline{\boldsymbol{R}})\,,
\end{equation}
i.e. $\,(\widehat{\boldsymbol{\varphi}}, \widehat{\boldsymbol{R}})\,$ is a minimizer of the energy functional $\,I\,$ over $\,\mathcal{A}\,$. The proof is complete.
\end{proof}

\begin{remark}\label{rem1}
In view of the relationship between the wryness tensor $\,\boldsymbol{\Gamma}\,$ and the dislocation density tensor $\,\overline{\boldsymbol{K}}$, described by \eqref{e9}, \eqref{e10} and \eqref{e13,5},  one can  express the curvature energy density alternatively as a function of $\,\boldsymbol{\Gamma}\,$ in the following way
    \begin{align}\label{e54}
      W=\widetilde{W}(\overline{\boldsymbol{E}},\boldsymbol{\Gamma})&=W_{\mathrm{mp}}(\overline{\boldsymbol{E}} )+ \widetilde{W}_{\mathrm{curv}}( \boldsymbol{\Gamma}),\qquad\mathrm{where} \vspace{4pt}\notag\\[-5pt]\\[-5pt]
       \widetilde{W}_{\mathrm{curv}}( \boldsymbol{\Gamma})&= \mu\,L_c^p\,\Big(b_1\|\,\mathrm{dev\,sym}\, \boldsymbol{\Gamma}\|^2\, +  \,b_2 \|\,\mathrm{skew}\, \boldsymbol{\Gamma}\|^2\,   + \, b_3\big(\mathrm{tr}\,\boldsymbol{\Gamma}\big)^2\Big)^{{p}/{2}}\,\,.\notag
    \end{align}
Taking into account \eqref{e13,5}, \eqref{e14}$_3$ and \eqref{e54}$_2\,$,  we see that
$$b_1=a_1\,,\qquad b_2=a_2\,,\qquad b_3=\,4\,a_3\,.$$
Then, one can formulate the minimization problem as in \eqref{e19}, \eqref{e21}, and prove in a similar manner the existence of minimizers under the same hypotheses \eqref{e15}, \eqref{e16}, \eqref{e22}, \eqref{e23}, with $\,b_1 ,\, b_2,\, b_3>0$.
\end{remark}

\begin{remark}\label{rem2}
The boundary conditions for the microrotation field $\,\overline{\boldsymbol{R}}\,$ on the Dirichlet part of the boundary $\,\Gamma_d\,$ can be relaxed or even omitted. Some alternative possible boundary conditions for the  microrotation  $\,\overline{\boldsymbol{R}}\,$ on   $\,\Gamma_d\,$ have been discussed in \cite{Neff_plate04_cmt,Neff_plate07_m3as}. Thus, we can consider (instead of \eqref{e21}) the larger admissible set $\,\widetilde{\mathcal{A}}\,$ given by
\begin{equation}\label{e55}
    \widetilde{\mathcal{A}}=\big\{\,(\boldsymbol{\varphi}, \overline{\boldsymbol{R}})\in H^1(\Omega, \mathbb{R}^3)\times W^{1,p}\big(\Omega,\mathrm{SO}(3)\big)\,\big|\,\,\, \boldsymbol{\varphi}_{\big|\Gamma_d}= \boldsymbol{\varphi}_d\,\,\big\}.
\end{equation}
In this situation, the existence theorem remains valid, in the following sense: the minimization problem \eqref{e19} and \eqref{e55} admits, under the hypotheses \eqref{e22} and \eqref{e23}$_1\,$, at least one solution $\,(\boldsymbol{\varphi}, \overline{\boldsymbol{R}})\in\widetilde{\mathcal{A}}$.
\end{remark}

\begin{remark}\label{rem3}
In this paper we have considered the case of isotropic materials. The existence results can also be generalized to the case of anisotropic elastic bodies, provided that the strain energy density $\,W(\overline{\boldsymbol{E}},\overline{\boldsymbol{K}})\,$ satisfies some appropriate coercivity and convexity conditions. The corresponding results for isotropic, orthotropic or anisotropic plates and shells of Cosserat type have been presented in \cite{Neff_plate04_cmt,Neff_plate07_m3as,Birsan-Neff-AnnRom-2012,Birsan-Neff-JElast-2013,Birsan-Neff-MMS-2014}.
\end{remark}

\section{Chiral materials}\label{sec5}

Let us exemplify our previous remark for chiral (or hemitropic) materials. A material is said to be chiral, if it cannot be superposed to its mirror image. Chiral materials are thus non-centro-symmetric due to the handedness in their microstructure. We refer the interested reader to \cite{Aero_Kuvshinski_63,Aero_Kuvshinski_64,natroshvili2005}.\\

While a material is isotropic, if it is right-invariant under the group $\mathrm{O}(3)$, chiral materials are characterized by right-invariance under $\mathrm{SO}(3)$-only. In general, the symmetry group being smaller, there exist more material constants. By straightforward extension of the linear chiral micropolar model we add three terms which couple non-symmetric strain $\overline{\boldsymbol{E}}$ and curvature $\overline{\boldsymbol{K}}$. We add for chiral materials (see \cite{Lakes1,Lakes2})
\begin{align}\label{e61}
	W_{\textrm{chiral}}(\overline{\boldsymbol{E}}\,,\,\overline{\boldsymbol{K}})=2\mu L_c\,\Big (\,&\sqrt{\beta_1}\,\langle\,\mathrm{dev\,sym}\, \overline{\boldsymbol{E}}\,,\,\mathrm{dev\,sym}\, \overline{\boldsymbol{K}}\,\rangle\notag\\
	\,+\,&\sqrt{\beta_2}\,\langle\,\mathrm{skew}\,\overline{\boldsymbol{E}}\,,\,\mathrm{skew}\,\overline{\boldsymbol{K}}\,\rangle\,+\,\sqrt{\beta_3}\,\mathrm{tr}\,\overline{\boldsymbol{E}}\,\,\mathrm{tr}\,\overline{\boldsymbol{K}}\,\Big)\,.
\end{align}

Let us consider the quadratic case $p=2$ in more detail. Then the total elastic energy reads
\begin{align}\label{eq111}
\begin{split}
	W(\overline{\boldsymbol{E}},\overline{\boldsymbol{K}})&=\mu\,\|\,\mathrm{dev\,sym}\, \overline{\boldsymbol{E}}\,\|^2\,\,+\,\mu_c\,\|\,\mathrm{skew}\, \overline{\boldsymbol{E}}\,\|^2\,+\,\frac{\kappa}{2}\big (\tr\,\overline{\boldsymbol{E}}\,\big)^2\\
	&\hspace{0.35cm}\,+2\mu L_c\,\Big (\,\sqrt{\beta_1}\,\langle\,\mathrm{dev\,sym}\, \overline{\boldsymbol{E}}\,,\,\mathrm{dev\,sym}\, \overline{\boldsymbol{K}}\,\rangle\\
	&\hspace{1.5cm}\,+\,\sqrt{\beta_2}\,\langle\,\mathrm{skew}\,\overline{\boldsymbol{E}}\,,\,\mathrm{skew}\,\overline{\boldsymbol{K}}\,\rangle\,+\,\sqrt{\beta_3}\,\mathrm{tr}\,\overline{\boldsymbol{E}}\,\,\mathrm{tr}\,\overline{\boldsymbol{K}}\,\Big)\\
	&\hspace{0.35cm}\,+\mu L_c^2\left(a_1\,\|\,\mathrm{dev\,sym}\,\overline{\boldsymbol{K}}\,\|^2\,+\,a_2\,\|\,\mathrm{skew}\,\overline{\boldsymbol{K}}\,\|^2+a_3\,\big (\tr\,\overline{\boldsymbol{K}}\,\big)^2\right)\,.
\end{split}
\end{align}
This formulation linearizes to 
\begin{align*}
\begin{split}
	W(\nabla \boldsymbol{u}, \overline{\boldsymbol{A}})&=\mu\,\|\,\mathrm{dev\,sym}\, \nabla\boldsymbol{u}\,\|^2\,\,+\,\mu_c\,\|\,\mathrm{skew}\, \big(\,\nabla\boldsymbol{u}-\overline{\boldsymbol{A}}\,\big)\,\|^2\,+\,\frac{\kappa}{2}\big (\tr\,\big(\,\nabla\boldsymbol{u}-\overline{\boldsymbol{A}}\,\big)\,\big)^2\\
	&\hspace{0.35cm}\,+2\mu L_c\,\Big (\,\sqrt{\beta_1}\,\langle\,\mathrm{dev\,sym}\, \nabla\boldsymbol{u}\,,\,\mathrm{dev\,sym}\, \mathrm{Curl}\,\overline{\boldsymbol{A}}\,\rangle\\
	&\hspace{1.5cm}\,+\,\sqrt{\beta_2}\,\langle\,\mathrm{skew}\,\nabla\boldsymbol{u}\,,\,\mathrm{skew}\,\mathrm{Curl}\,\overline{\boldsymbol{A}}\,\rangle\,+\,\sqrt{\beta_3}\,\mathrm{tr}\,\nabla\boldsymbol{u}\,\,\mathrm{tr}\,\mathrm{Curl}\,\overline{\boldsymbol{A}}\,\Big)\\
	&\hspace{0.35cm}\,+\mu L_c^2\left(a_1\,\|\,\mathrm{dev\,sym}\,\mathrm{Curl}\,\overline{\boldsymbol{A}}\,\|^2\,+\,a_2\,\|\,\mathrm{skew}\,\mathrm{Curl}\,\overline{\boldsymbol{A}}\,\|^2+a_3\,\big (\tr\,\mathrm{Curl}\,\overline{\boldsymbol{A}}\,\big)^2\right)\,,
\end{split}
\end{align*}
which is equivalent to the format of linear chiral Cosserat materials given e.\,g. in \cite{natroshvili2005}, (2.19). Note that \eqref{eq111} is the most general chiral quadratic energy in terms of $\big(\,\overline{\boldsymbol{E}},\overline{\boldsymbol{K}}\,\big)$ having nine material constants. In order to establish positive definiteness of this energy with respect to $\big(\,\overline{\boldsymbol{E}},\overline{\boldsymbol{K}}\,\big)$ we observe that the new chiral terms have no sign in general and must therefore be bounded by the already present strain and curvature contributions. A sample set of sufficient conditions can be obtained from the following considerations:
\begin{align*}
&\mu\,\|\,\mathrm{dev\,sym}\,\overline{\boldsymbol{E}}\,\|^2\,+\,\mu L_c^2a_1\,\|\,\mathrm{dev\,sym}\,\overline{\boldsymbol{K}}\,\|^2\,+\,2\mu L_c\sqrt{\beta_1}\,\langle\,\mathrm{dev\,sym}\, \overline{\boldsymbol{E}}\,,\,\mathrm{dev\,sym}\, \overline{\boldsymbol{K}}\,\rangle\\
	&\hspace*{0.25cm}\geq\,\mu\,\|\,\mathrm{dev\,sym}\,\overline{\boldsymbol{E}}\,\|^2+\mu L_c^2 a_1\,\|\,\mathrm{dev\,sym}\,\overline{\boldsymbol{K}}\,\|^2\,-\,2\mu L_c\sqrt{\beta_1}\,\|\,\mathrm{dev\,sym}\,\overline{\boldsymbol{E}}\,\|\,\|\,\mathrm{dev\,sym}\, \overline{\boldsymbol{K}}\,\|\,\\
	&\hspace{0.5cm}=\langle\,\left(\begin{array}{c}\|\,\mathrm{dev\,sym}\,\overline{\boldsymbol{E}}\,\|\\[12pt]
	\|\,\mathrm{dev\,sym}\,\overline{\boldsymbol{K}}\,\|\end{array}\right)\,,\,\underbrace{\begin{pmatrix}\displaystyle\mu & -\displaystyle\mu L_c\sqrt{\beta_1}\\[12pt]\displaystyle-\mu L_c\sqrt{\beta_1} & \displaystyle\mu L_c^2a_1\end{pmatrix}}_{=M}\left(\begin{array}{c}\|\,\mathrm{dev\,sym}\,\overline{\boldsymbol{E}}\,\|\\[12pt]
	\|\,\mathrm{dev\,sym}\,\overline{\boldsymbol{K}}\,\|\end{array}\right)\,\rangle\,.
\end{align*}
Since the matrix M is symmetric, the last term is positive if and only if 
\begin{align*}
	\mu>0\quad\text{and}\;\;\,\det(M)=\mu^2L_c^2\,\big(\,a_1-\beta_1\big)>0\qquad\Longleftrightarrow\qquad a_1>\beta_1.
\end{align*}

Proceeding similar by the other two terms we must have in addition to conditions \eqref{e15} and \eqref{e16} that
\begin{alignat}{3}\label{e62}
	a_1>\beta_1\geq 0,\qquad a_2>\frac{\mu}{\mu_c}\beta_2\geq 0,\qquad\text{and}\qquad a_3>\frac{\mu}{\kappa}\,\beta_3\geq 0
\end{alignat}
In summary, pointwise uniform positive definiteness of the energy \eqref{eq111} is obtained if and only if
\begin{alignat}{4}\label{e63}
	\mu&>0,\qquad				   &\kappa&>0,\qquad								&\mu_c&>0,\qquad 				& 	p&\geq 2,\notag\\&&&&&&&\\[-5pt]
	a_1&>\beta_1\geq 0,\;\;\qquad    &     a_2&>\frac{\mu}{\mu_c}\,\beta_2\geq 0,\;\;\qquad		&    a_3&>\frac{\mu}{\kappa}\,\beta_3\geq 0,\;\;\qquad 	&  L_c&>0,\notag
\end{alignat}
See (2.19) in \cite{natroshvili2005} for a more complicated, equivalent statement in linear chiral Cosserat models. With the set of assumptions \eqref{e63} our existence result carries over.

\section{Limitations of our approach}\label{sec6}

Another coupling between micro and continuum rotations can be envisaged. Instead of taking
\begin{align}\label{e64}
	\mu_c\,\|\,\mathrm{skew}\,\big(\,\overline{\boldsymbol{U}}-\id_3\big)\,\|^2=\mu_c\,\|\,\mathrm{skew}\,\big(\,\overline{\boldsymbol{R}}^T\boldsymbol{R}\,\boldsymbol{U}\big)\,\|^2\,,
\end{align}
which mixes rotational coupling with pure stretch $\boldsymbol{U}$, we might consider
\begin{align}\label{e65}
	\mu_c\,\|\,\overline{\boldsymbol{R}}-\boldsymbol{R}\,\|^2=\mu_c\,\|\,\overline{\boldsymbol{R}}^T\boldsymbol{R}-\id_3\,\|^2=\mu_c\,\left\|\,\overline{\boldsymbol{U}}\left(\overline{\boldsymbol{U}}^T\overline{\boldsymbol{U}}\;\right)^{-\frac{1}{2}}-\id_3\,\right\|^2\,,
\end{align}
where $\boldsymbol{R}$ is the continuum rotation (the orthogonal factor in the polar decomposition $\boldsymbol{F}=\boldsymbol{R}\,\boldsymbol{U}$ ). Both terms have the identical linearized behaviour. However, \[\boldsymbol{X}\longmapsto\left\|\,\boldsymbol{X}\left(\boldsymbol{X}^T\boldsymbol{X}\;\right)^{-\frac{1}{2}}-\id_3\,\right\|^2\] is non-convex with respect to $\boldsymbol{X}$ and our methods do not immediately carry over. Nevertheless, the coupling \eqref{e65} is used in \cite{BohmNe_14} giving rise to solitary wave solutions.\\

Similary, we could use the Riemannian distance function on $\mathrm{SO}(3)$, the square of which is given by
\begin{align}
	\mu_c\,\mathrm{dist}_{\mathrm{geod}}^2\Big(\overline{\boldsymbol{R}}^T\overline{\boldsymbol{R}}\big)=\mu_c\,\|\,\log\big(\,\overline{\boldsymbol{R}}^T\boldsymbol{R}\,\big)\,\|^2=\mu_c\,\left \|\,\log\left(\,\overline{\boldsymbol{U}}\left(\overline{\boldsymbol{U}}^T\overline{\boldsymbol{U}}\;\right)^{-\frac{1}{2}}\,\right)\,\right\|^2\,,
\end{align}
where $\log$ is the matrix-logarithm \cite{Neff_Martin}. In this case it is again the non-convexity which prevents application of our method.\\

All coupling terms have the same linearized behaviour namely
\begin{align}
	\mu_c\,\|\,\mathrm{skew}\,(\nabla\boldsymbol{u}-\overline{\boldsymbol{A}})\,\|^2\,, 
\end{align}
where $\mathrm{skew}\,\nabla u$ is the infinitesimal continuum rotation and $\overline{\boldsymbol{A}}$ is the infinitesimal microrotation.
\footnotesize{\bibliographystyle{plain}
\bibliography{literatur_Birsan}}

\appendix
\section*{Appendix A. Other curvature measures}\label{app1}

Instead of taking $\boldsymbol{\mathfrak{K}}$ one can also work with $\boldsymbol{\widetilde{\mathfrak{K}}}=\boldsymbol{\mathfrak{K}}^{\stackrel{2.3}{T}}$. This approach for example can be found in \cite{Boehm14, Boehm13}, where $ \boldsymbol{\widetilde{\mathfrak{K}}}$ is named $\boldsymbol{K}$. 
Therefore we find it expedient to provide a summary table which shows how tensors used in this paper transform.\\ 

According to definition \eqref{e2}, \eqref{e3} and \eqref{e7} we have
\begin{align}
	\boldsymbol{\mathfrak{K}}&=\overline{\boldsymbol{R}}^T\mathrm{Grad}\,\overline{\boldsymbol{R}}=(\overline{\boldsymbol{R}}^T\overline{\boldsymbol{R}},_k)\otimes\boldsymbol{e}_k=\overline{R}_{mi}\,\overline{R}_{mj,k}\,\boldsymbol{e}_i\otimes\boldsymbol{e}_j\otimes\boldsymbol{e}_k\,,\label{e65}\\
	\boldsymbol{\widetilde{\mathfrak{K}}}&=\boldsymbol{\mathfrak{K}}^{\stackrel{2.3}{T}}=\overline{R}_{mi}\,\overline{R}_{mk,j}\,\boldsymbol{e}_i\otimes\boldsymbol{e}_j\otimes\boldsymbol{e}_k\,,\label{e66}
\end{align}
and
\begin{align}\label{e67}
	\overline{\boldsymbol{K}}=\overline{\boldsymbol{R}}^T\mathrm{Curl}\,\overline{\boldsymbol{R}}=-\overline{\boldsymbol{R}}^T(\,\overline{\boldsymbol{R}},_i\times \boldsymbol{e}_i\,)=\epsilon_{jrs}\,\overline{R}_{mi}\,\overline{R}_{ms,r}\,\boldsymbol{e}_i\otimes\boldsymbol{e}_j.
\end{align}
Furthermore equation \eqref{e4} yields
\begin{align}\label{e68}
\Gamma&=\mathrm{axl}(\,\overline{\boldsymbol{R}}^T\overline{\boldsymbol{R}},_k)\otimes\boldsymbol{e}_k
	      =-\frac{1}{2}\,\epsilon_{ijk}\,\overline{R}_{mj}\,\overline{R}_{mk,t}\,\boldsymbol{e}_i\otimes\boldsymbol{e}_t
\end{align} 
where the double dot product `` : '' of two third order tensors $\boldsymbol{A}=A_{ijk}\,\boldsymbol{e}_i\otimes\boldsymbol{e}_j\otimes \boldsymbol{e}_k$ and $\boldsymbol{B}=B_{ijk}\,\boldsymbol{e}_i\otimes\boldsymbol{e}_j\otimes \boldsymbol{e}_k$
is again defined as $\,\boldsymbol{A}:\boldsymbol{B}\,=\, A_{irs}B_{rsj}\,\boldsymbol{e}_i\otimes\boldsymbol{e}_j\,$.
Then we can easily derive an relationship between the third order curvature tensor $\boldsymbol{\mathfrak{K}}$ and the dislocation density tensor $\overline{\boldsymbol{K}}$:
\begin{align}
	\boldsymbol{\mathfrak{K}}\,:\,\boldsymbol{\epsilon}&=(\,{\mathfrak{K}}_{ilm}\,\boldsymbol{e}_i\otimes\boldsymbol{e}_l\otimes\boldsymbol{e}_m\,)\,:\,(\,\epsilon_{jkn}\,\boldsymbol{e}_j\otimes\boldsymbol{e}_k\otimes\boldsymbol{e}_n\,)=\mathfrak{K}_{irs}\,\epsilon_{rsn}\,\boldsymbol{e}_i\otimes\boldsymbol{e}_n\notag\\
		&=\overline{R}_{ji}\,\overline{R}_{jr,s}\,\epsilon_{rsn}\,\boldsymbol{e}_i\otimes\boldsymbol{e}_n=\epsilon_{nrs}\,\overline{R}_{ji}\,\overline{R}_{jr,s}\,\boldsymbol{e}_i\otimes\boldsymbol{e}_n=-\overline{\boldsymbol{K}}\,,
\end{align}
which means
\begin{align}\label{e69}
	\overline{\boldsymbol{K}}=-\boldsymbol{\mathfrak{K}}:\boldsymbol{\epsilon}\,.
\end{align}
On the other hand we have
\begin{align}
	\boldsymbol{K}:\boldsymbol{\epsilon}&=\boldsymbol{\widetilde{\mathfrak{K}}}\,:\,\boldsymbol{\epsilon}=(\,{\mathfrak{K}}_{iml}\,\boldsymbol{e}_i\otimes\boldsymbol{e}_l\otimes\boldsymbol{e}_m\,)\,:\,(\,\epsilon_{jkn}\,\boldsymbol{e}_j\otimes\boldsymbol{e}_k\otimes\boldsymbol{e}_n\,)\notag\\
				  &=\mathfrak{K}_{isr}\,\epsilon_{rsn}\,\boldsymbol{e}_i\otimes\boldsymbol{e}_n=\overline{R}_{ji}\,\overline{R}_{js,r}\,\epsilon_{rsn}\,\boldsymbol{e}_i\otimes\boldsymbol{e}_n=\epsilon_{nrs}\,\overline{R}_{ji}\,\overline{R}_{js,r}\,\boldsymbol{e}_i\otimes\boldsymbol{e}_n=\overline{\boldsymbol{K}}\,,
\end{align}
such that
\begin{align}\label{e70}
	\overline{\boldsymbol{K}}=\overline{\boldsymbol{R}}^T\mathrm{Curl}\,\boldsymbol{R}=\boldsymbol{K}:\boldsymbol{\epsilon}
\end{align}
Moreover we have 
\begin{align}
	\boldsymbol{R}^T\boldsymbol{R}=\id\;&\Longrightarrow\;\partial_j\,\big(\,\boldsymbol{R}^T\boldsymbol{R}\,\big)=\big(\,\partial_j\boldsymbol{R}^T\,\big)\boldsymbol{R}+\boldsymbol{R}^T\big(\,\partial_j\boldsymbol{R}\,\big)=0\notag\\
	&\Longrightarrow\;\big(\,\partial_j\boldsymbol{R}^T\,\big)\boldsymbol{R}=-\boldsymbol{R}^T\big(\,\partial_j\boldsymbol{R}\,\big)
\end{align}
and therefore
\begin{align}
	\boldsymbol{\epsilon}\boldsymbol{\Gamma}&=\big(\,\epsilon_{ijk}\,\boldsymbol{e}_i\otimes\boldsymbol{e}_j\otimes\boldsymbol{e}_k\,\big)\left(-\frac{1}{2}\,\epsilon_{rst}\,\overline{R}_{ms}\,\overline{R}_{mt,n}\,\boldsymbol{e}_r\otimes\boldsymbol{e}_n\right)
						=-\frac{1}{2}\,\epsilon_{ijk}\,\epsilon_{kst}\,\overline{R}_{ms}\,\overline{R}_{mt,n}\,\boldsymbol{e}_i\otimes\boldsymbol{e}_j\otimes\boldsymbol{e}_n\notag\\
						&=\frac{1}{2}\,\epsilon_{ikj}\,\epsilon_{kst}\,\overline{R}_{ms}\,\overline{R}_{mt,n}\,\boldsymbol{e}_i\otimes\boldsymbol{e}_j\otimes\boldsymbol{e}_n=\frac{1}{2}\,\big(\,\delta_{it}\delta_{js}-\delta_{is}\delta_{jt}\,\big)\,\overline{R}_{ms}\,\overline{R}_{mt,n}\,\boldsymbol{e}_i\otimes\boldsymbol{e}_j\otimes\boldsymbol{e}_n\notag\\
						&=\frac{1}{2}\,\overline{R}_{mj}\,\overline{R}_{mi,n}\,\boldsymbol{e}_i\otimes\boldsymbol{e}_j\otimes\boldsymbol{e}_n-\frac{1}{2}\,\overline{R}_{mi}\,\overline{R}_{mj,n}\,\boldsymbol{e}_i\otimes\boldsymbol{e}_j\otimes\boldsymbol{e}_n\\
						&=-\frac{1}{2}\,\overline{R}_{mi}\,\overline{R}_{mj,n}\,\boldsymbol{e}_i\otimes\boldsymbol{e}_j\otimes\boldsymbol{e}_n-\frac{1}{2}\,\overline{R}_{mi}\,\overline{R}_{mj,n}\,\boldsymbol{e}_i\otimes\boldsymbol{e}_j\otimes\boldsymbol{e}_n\notag\\
						&=-\overline{R}_{mi}\,\overline{R}_{mj,n}\,\boldsymbol{e}_i\otimes\boldsymbol{e}_j\otimes\boldsymbol{e}_n=-\boldsymbol{\mathfrak{K}}\,,\notag
\end{align}
 thus
\begin{align}\label{e71}
	\boldsymbol{\mathfrak{K}}=-\boldsymbol{\epsilon}\boldsymbol{\Gamma}\,.
\end{align}
In \cite{Boehm1} the curvature is also characterized by the \emph{Cartan torsion tensor} (cf. \cite{Hehl07}) $\boldsymbol{\mathfrak{T}}=\boldsymbol{\widetilde{\mathfrak{K}}}-\boldsymbol{\mathfrak{K}}$, which is also easily connected to the dislocation density tensor through
\begin{align}\label{e72}
	\boldsymbol{\mathfrak{T}}=\boldsymbol{\widetilde{\mathfrak{K}}}-\boldsymbol{\mathfrak{K}}=\overline{\boldsymbol{K}}\boldsymbol{\epsilon}\,,
\end{align}
since
\begin{align}
	\overline{\boldsymbol{K}}\boldsymbol{\epsilon}&=\big(\,\epsilon_{jrs}\,\overline{R}_{li}\,\overline{R}_{ls,r}\,\boldsymbol{e}_i\otimes\boldsymbol{e}_j\,\big)\,\big(\,\epsilon_{kmn}\,\boldsymbol{e}_k\otimes\boldsymbol{e}_m\otimes\boldsymbol{e}_n\,\big)
					=\epsilon_{krs}\,\epsilon_{kmn}\,\overline{R}_{li}\,\overline{R}_{ls,r}\,\boldsymbol{e}_i\otimes\boldsymbol{e}_m\otimes\boldsymbol{e}_n\notag\\
					&=-\epsilon_{rks}\,\epsilon_{kmn}\,\overline{R}_{li}\,\overline{R}_{ls,r}\,\boldsymbol{e}_i\otimes\boldsymbol{e}_m\otimes\boldsymbol{e}_n
					=-\big(\,\delta_{rn}\delta_{sm}-\delta_{rm}\delta_{sn}\,\big)\,\overline{R}_{li}\,\overline{R}_{ls,r}\,\boldsymbol{e}_i\otimes\boldsymbol{e}_m\otimes\boldsymbol{e}_n\\
					&=\big(\,\delta_{rm}\delta_{sn}-\delta_{rn}\delta_{sm}\,\big)\,\overline{R}_{li}\,\overline{R}_{ls,r}\,\boldsymbol{e}_i\otimes\boldsymbol{e}_m\otimes\boldsymbol{e}_n\notag\\
					&=\overline{R}_{li}\,\overline{R}_{ln,m}\,\boldsymbol{e}_i\otimes\boldsymbol{e}_m\otimes\boldsymbol{e}_n-\overline{R}_{li}\,\overline{R}_{lm,n}\,\boldsymbol{e}_i\otimes\boldsymbol{e}_m\otimes\boldsymbol{e}_n=\boldsymbol{\widetilde{\mathfrak{K}}}-\boldsymbol{\mathfrak{K}}=\boldsymbol{\mathfrak{T}}\,.\notag
\end{align}
Then with \eqref{e72}, \eqref{e71} and \eqref{e9}$_2$ we obtain
\begin{align}\label{e73}
	\boldsymbol{K}=\boldsymbol{\widetilde{\mathfrak{K}}}=\boldsymbol{\mathfrak{T}}+\boldsymbol{\mathfrak{K}}=\overline{\boldsymbol{K}}\boldsymbol{\epsilon}-\boldsymbol{\epsilon}\boldsymbol{\Gamma}=\overline{\boldsymbol{K}}\boldsymbol{\epsilon}+\boldsymbol{\epsilon}\overline{\boldsymbol{K}}^T-\frac{1}{2}\big(\mathrm{tr}\,\overline{\boldsymbol{K}}\,\big)\boldsymbol{\epsilon}
\end{align}
Furthermore we have
\begin{align}
	\boldsymbol{\mathfrak{T}}^{\stackrel{2.3}{T}}&=\big(\,\boldsymbol{\widetilde{\mathfrak{K}}}-\boldsymbol{\mathfrak{K}}\,\big)^{\stackrel{2.3}{T}}=\boldsymbol{\widetilde{\mathfrak{K}}}^{\stackrel{2.3}{T}}-\boldsymbol{\mathfrak{K}}^{\stackrel{2.3}{T}}=\boldsymbol{\mathfrak{K}}-\boldsymbol{\widetilde{\mathfrak{K}}}\,,\\
	\boldsymbol{\mathfrak{T}}^{\stackrel{1.3}{T}}&=\big(\,\boldsymbol{\widetilde{\mathfrak{K}}}-\boldsymbol{\mathfrak{K}}\,\big)^{\stackrel{1.3}{T}}=\boldsymbol{\widetilde{\mathfrak{K}}}^{\stackrel{1.3}{T}}-\boldsymbol{\mathfrak{K}}^{\stackrel{1.3}{T}}=-\boldsymbol{\widetilde{\mathfrak{K}}}-\boldsymbol{\mathfrak{K}}^{\stackrel{1.3}{T}}\,,
\end{align}
and
\begin{align}
	\boldsymbol{\mathfrak{T}}^{\stackrel{1.2}{T}}&=\big(\,\boldsymbol{\widetilde{\mathfrak{K}}}-\boldsymbol{\mathfrak{K}}\,\big)^{\stackrel{1.2}{T}}
	=\big(\overline{R}_{mi}\,\overline{R}_{mk,j}-\overline{R}_{mi}\,\overline{R}_{mj,k}\,\big)\,\boldsymbol{e}_j\otimes\boldsymbol{e}_i\otimes\boldsymbol{e}_k\notag\\
	&=\big(-\overline{R}_{mk}\,\overline{R}_{mi,j}+\overline{R}_{mj}\,\overline{R}_{mi,k}\,\big)\,\boldsymbol{e}_j\otimes\boldsymbol{e}_i\otimes\boldsymbol{e}_k\\
	&=-\overline{R}_{mk}\,\overline{R}_{mi,j}\,\boldsymbol{e}_j\otimes\boldsymbol{e}_i\otimes\boldsymbol{e}_k+\overline{R}_{mj}\,\overline{R}_{mi,k}\,\boldsymbol{e}_j\otimes\boldsymbol{e}_i\otimes\boldsymbol{e}_k=-\boldsymbol{\mathfrak{K}}^{\stackrel{1.3}{T}}+\boldsymbol{\mathfrak{K}}\,,\notag
\end{align}
thus
\begin{align}
	\boldsymbol{\mathfrak{T}}+\boldsymbol{\mathfrak{T}}^{\stackrel{1.2}{T}}-\boldsymbol{\mathfrak{T}}^{\stackrel{1.3}{T}}=2\cdot\boldsymbol{\widetilde{\mathfrak{K}}}
\end{align}
Now
\begin{align}
	\boldsymbol{\mathfrak{T}}:\boldsymbol{\epsilon}&=\big(T_{ijk}\,\boldsymbol{e}_i\otimes\boldsymbol{e}_j\otimes\boldsymbol{e}_k\big):\big(\epsilon_{rst}\boldsymbol{e}_r\otimes\boldsymbol{e}_s\otimes\boldsymbol{e}_t\big)\notag\\
	&=T_{ijk}\,\epsilon_{jkt}\,\boldsymbol{e}_i\otimes\boldsymbol{e}_t=\epsilon_{jkt}\,\big(\overline{R}_{mi}\,\overline{R}_{mk,j}-\overline{R}_{mi}\,\overline{R}_{mj,k}\,\big)\,\boldsymbol{e}_i\otimes\boldsymbol{e}_t\\
	&=\epsilon_{tjk}\,\overline{R}_{mi}\,\overline{R}_{mk,j}\,\boldsymbol{e}_i\otimes\boldsymbol{e}_t-\big(-\epsilon_{tkj}\overline{R}_{mi}\,\overline{R}_{mj,k}\,\boldsymbol{e}_i\otimes\boldsymbol{e}_t\big)
	=\overline{\boldsymbol{K}}+\overline{\boldsymbol{K}}=2\,\overline{\boldsymbol{K}}\,.\notag
\end{align}
\begin{landscape}
In summary we have the following relationships between the second Cosserat deformation tensor $\boldsymbol{\mathfrak{K}}=\overline{\boldsymbol{R}}^T\mathrm{Grad}\,\overline{\boldsymbol{R}}$, the wryness-tensor $\boldsymbol{\Gamma}=\mathrm{axl}(\,\overline{\boldsymbol{R}}^T\overline{\boldsymbol{R}},_k)\otimes\boldsymbol{e}_k$, the contortion tensor $\boldsymbol{\widetilde{\mathfrak{K}}}= \overline{\boldsymbol{R}}^T\big(\, \mathrm{Grad}\, \boldsymbol{d }_k \big) \otimes \, \boldsymbol{e}_k $, the dislocation density tensor $\overline{\boldsymbol{K}}=\overline{\boldsymbol{R}}^T\mathrm{Curl}\,\overline{\boldsymbol{R}}$ and the Cartan torsion tensor $\boldsymbol{\mathfrak{T}}=\boldsymbol{\widetilde{\mathfrak{K}}}-\boldsymbol{\mathfrak{K}}$:\\
\begin{table}[h]
\centering
\renewcommand{\arraystretch}{1.27}
\begin{tabular}{C{1.6cm}|C{3.25cm}|C{4.15cm}|C{5.2cm}|C{2.85cm}|C{2.85cm}}
\diagbox[height=1.5cm, width=2cm]{from}{to}			
												& $\boldsymbol{\mathfrak{K}}$ 										& $\boldsymbol{\widetilde{\mathfrak{K}}}=\boldsymbol{K}$ 	& $\boldsymbol{\Gamma}$ 
												& $\overline{\boldsymbol{K}}$ 											& $\boldsymbol{\mathfrak{T}}$																	\\\hline
												&&&&&																																			\\
$\boldsymbol{\mathfrak{K}}$							& -																& $\boldsymbol{\mathfrak{K}}^{\stackrel{2.3}{T}}$		   	
												& $\big(\boldsymbol{\mathfrak{K}}:\boldsymbol{\epsilon}\,\big)^T-\displaystyle\frac{1}{2}\,\big(\mathrm{tr}\,\boldsymbol{\mathfrak{K}}:\boldsymbol{\epsilon}\,\big)\,\id_3$	
												& $-\boldsymbol{\mathfrak{K}}:\boldsymbol{\epsilon}$ 						& $\boldsymbol{\mathfrak{K}}^{\stackrel{2.3}{T}}-\boldsymbol{\mathfrak{K}}$						\\
												&&&&&																																			\\\hline
												&&&&&																																			\\
$\boldsymbol{\widetilde{\mathfrak{K}}}=\boldsymbol{K}$		& $\boldsymbol{\widetilde{\mathfrak{K}}}^{\stackrel{2.3}{T}}$					& -												
												& $\big(\boldsymbol{\widetilde{\mathfrak{K}}}^{\stackrel{2.3}{T}}\hspace*{-7pt}:\boldsymbol{\epsilon}\,\big)^T-\displaystyle\frac{1}{2}\,\big(\mathrm{tr}\,\boldsymbol{\widetilde{\mathfrak{K}}}^{\stackrel{2.3}{T}}\hspace*{-7pt}:\boldsymbol{\epsilon}\,\big)\,\id_3$		
												& $-\boldsymbol{\widetilde{\mathfrak{K}}}^{\stackrel{2.3}{T}}\hspace*{-7pt}:\boldsymbol{\epsilon}$			& $\boldsymbol{\widetilde{\mathfrak{K}}}-\boldsymbol{\widetilde{\mathfrak{K}}}^{\stackrel{2.3}{T}}$																																				\\
												&&&&&																																			\\\hline
												&&&&&																																			\\
$\boldsymbol{\Gamma}$								& $-\boldsymbol{\epsilon}\boldsymbol{\Gamma}$							
												& $-\boldsymbol{\Gamma}^T\boldsymbol{\epsilon}-\boldsymbol{\epsilon}\boldsymbol{\Gamma}+\big(\mathrm{tr}\,\boldsymbol{\Gamma}\,\big)\,\boldsymbol{\epsilon}$						
												& -						
												& $-\boldsymbol{\Gamma}^T+\big(\mathrm{tr}\,\boldsymbol{\Gamma}\,\big)\,\id_3$	& $-\boldsymbol{\Gamma}^T\boldsymbol{\epsilon}+\big(\mathrm{tr}\,\boldsymbol{\Gamma}\,\big)\,\boldsymbol{\epsilon}$																																\\
												&&&&&																																			\\\hline
												&&&&&																																			\\ 
$\overline{\boldsymbol{K}}$							& $\boldsymbol{\epsilon}\overline{\boldsymbol{K}}^T-\frac{1}{2}\big(\mathrm{tr}\,\overline{\boldsymbol{K}}\,\big)\,\boldsymbol{\epsilon}$	
												& $\overline{\boldsymbol{K}}\boldsymbol{\epsilon}+\boldsymbol{\epsilon}\overline{\boldsymbol{K}}^T-\displaystyle\frac{1}{2}\,\big(\mathrm{tr}\,\overline{\boldsymbol{K}}\,\big)\boldsymbol{\epsilon}$ 																											
												& $-\overline{\boldsymbol{K}}^T+\displaystyle\frac{1}{2}\,\big(\mathrm{tr}\,\overline{\boldsymbol{K}}\big)\,\id_3$				
												& -																& $\overline{\boldsymbol{K}}\boldsymbol{\epsilon}$											\\
												&&&&&																																			\\\hline
												&&&&&																																			\\
$\boldsymbol{\mathfrak{T}}$							& $\displaystyle\frac{1}{2}\,\big(\,\boldsymbol{\mathfrak{T}}^{\stackrel{1.2}{T}}\hspace*{-7pt}-\boldsymbol{\mathfrak{T}}-\boldsymbol{\mathfrak{T}}^{\stackrel{1.3}{T}}\,\big)$					
												& $\displaystyle\frac{1}{2}\,\big(\,\boldsymbol{\mathfrak{T}}+\boldsymbol{\mathfrak{T}}^{\stackrel{1.2}{T}}\hspace*{-7pt}-\boldsymbol{\mathfrak{T}}^{\stackrel{1.3}{T}}\,\big)$ 					
												& $-\displaystyle\frac{1}{2}\,\big(\boldsymbol{\mathfrak{T}}:\boldsymbol{\epsilon}\big)^T+\displaystyle\frac{1}{4}\,\big(\mathrm{tr}\,\boldsymbol{\mathfrak{T}}:\boldsymbol{\epsilon}\big)\,\id_3$	
												& $\displaystyle\frac{1}{2}\,\boldsymbol{\mathfrak{T}}:\boldsymbol{\epsilon}$		& -																		\\
												&&&&&											
\end{tabular}
\end{table}

Here $\boldsymbol{\mathfrak{K}}^{^{\stackrel{2.3}{T}}}$ is created by taking the transpose in the last two components of $\boldsymbol{\mathfrak{K}}$.
\end{landscape}

\end{document}